%% file: RKHS-TO.tex
\documentclass
[
    a4paper,
    DIV=11,
    abstracton,
    numbers=noendperiod
]
{scrartcl}

\usepackage
{
    graphicx,
    amssymb,
    amsmath,
    amsthm,
    dsfont, % for one vector
    xcolor,
    enumitem,
    tikz-cd, % for commutative diagrams
    mathtools,
    natbib,
    ifthen,
    enumitem,
    authblk,
    csquotes
}

\usepackage[bf,normal]{caption}

\usepackage[pdffitwindow=false,
            plainpages=false,
            pdfpagelabels=true,
            pdfpagemode=UseOutlines,
            pdfpagelayout=SinglePage,
            bookmarks=false,
            colorlinks=true,
            hyperfootnotes=false,
            linkcolor=blue,
            urlcolor=blue!30!black,
            citecolor=green!50!black]{hyperref}

\newtheorem{theorem}{Theorem}[section]
\newtheorem{corollary}[theorem]{Corollary}

\newtheorem{proposition}[theorem]{Proposition}
\newtheorem{definition}[theorem]{Definition}
\theoremstyle{definition}
\newtheorem{example}[theorem]{Example}
\newtheorem{remark}[theorem]{Remark}

\input{RKHS-TO-macros.tex}

\title{Eigendecompositions of Transfer Operators in \\ Reproducing Kernel Hilbert Spaces}

\author[1]{Stefan Klus}
\author[2]{Ingmar Schuster}
\author[3]{Krikamol Muandet}
\affil[1]{\normalsize Department of Mathematics and Computer Science, Freie Universit\"at Berlin, Germany}
\affil[2]{\normalsize Zalando Research, Zalando SE, Berlin, Germany}
\affil[3]{\normalsize Max Planck Institute for Intelligent Systems, T\"ubingen, Germany}
\date{}

\begin{document}

\maketitle

\begin{abstract}
Transfer operators such as the Perron--Frobenius or Koopman operator play an important role in the global analysis of complex dynamical systems. The eigenfunctions of these operators can be used to detect metastable sets, to project the dynamics onto the dominant slow processes, or to separate superimposed signals. We propose \emph{kernel transfer operators}, which extend transfer operator theory to reproducing kernel Hilbert spaces and show that these operators are related to Hilbert space representations of conditional distributions, known as conditional mean embeddings. The proposed numerical methods to compute empirical estimates of these kernel transfer operators subsume existing data-driven methods for the approximation of transfer operators such as extended dynamic mode decomposition and its variants. One main benefit of the presented kernel-based approaches is that they can be applied to any domain where a similarity measure given by a kernel is available. Furthermore, we provide elementary results on eigendecompositions of finite-rank RKHS operators. We illustrate the results with the aid of guiding examples and highlight potential applications in molecular dynamics as well as video and text data analysis.
\end{abstract}

%\begin{keywords}
%  Transfer operators, covariance operators, kernel mean embeddings, eigendecompositions, data-driven methods
%\end{keywords}

\section{Introduction}

Transfer operators such as the Perron--Frobenius or Koopman operator are ubiquitous in molecular dynamics, fluid dynamics, atmospheric sciences, and control theory \citep{Schmid10, BBPK16, KNKWKSN18}. The eigenfunctions of these operators can be used to decompose a system given by an ergodic Markov process into fast and slow dynamics and to identify modes of the stationary measure, in the molecular dynamics context called metastable sets, corresponding to conformations of molecules. The methods presented in this paper can be applied to data generated by any \emph{nonlinear} dynamical system and we will show potential novel applications pertaining to video and text data analysis. In molecular dynamics, we are in particular interested in the slow conformational changes of molecules and the corresponding transition probabilities and transition paths.

Over the last decades, different numerical methods such as \emph{Ulam's method} \citep{Ulam60}, \emph{extended dynamic mode decomposition} (EDMD) \citep{WKR15, WRK15, KKS16}, the \emph{variational approach of conformation dynamics} (VAC) \citep{NoNu13, NKPMN14}, and several extensions and generalizations have been developed to approximate transfer operators and their eigenvalues and eigenfunctions. The advantage of purely data-driven methods is that they can be applied to simulation or measurement data. Hence, information about the underlying system itself is not required. An overview and comparison of such methods can be found in~\citet{KNKWKSN18}. Applications and variants of these methods are also described in \citet{RMBSH09, TRLBK14, MP15}. Kernel-based reformulations of the aforementioned methods have been proposed in~\citet{WRK15} and \citet{SP15}.

In this work, we construct representations of transfer operators using reproducing kernel Hilbert space (RKHS) theory. We can directly express the kernel transfer operators in terms of covariance and cross-covariance operators in the RKHS. The benefits of kernel-based methods are twofold: First, the basis functions need not be defined explicitly, which thereby allows us to handle infinite-dimensional feature spaces. Second, the proposed methods can not only be applied to dynamical systems defined on Euclidean spaces, but also to systems defined on any domain that admits an appropriate kernel function such as images, graphs, or strings. We show that the kernel transfer operators are closely related to recently developed Hilbert space embeddings of probability distributions \citep{Berlinet04:RKHS, Smola07Hilbert, MFSS16}. Moreover, we propose an eigendecomposition technique for finite-rank operators acting on an RKHS. We show that the eigenfunctions belong to the RKHS associated with the kernel and can be expressed entirely in terms of the eigenvectors and eigenvalues of Gram matrices defined for training data. Therefore, our technique resembles several existing kernel-based component analysis techniques in machine learning. For example, kernel principal component analysis (KPCA) and kernel canonical correlation analysis (KCCA) extend the well-known PCA and CCA to data mapped into an RKHS \citep{Scholkopf98:KPCA,Bach03:KICA,Fukumizu07:KCCA}. In fact, KPCA can be seen the application of our eigendecomposition results to a particular RKHS operator.

Our work provides a unified framework for approximating transfer operators and their eigenfunctions. Given that dynamical systems are ubiquitous in machine learning, this leads to novel applications such as visualization of high-dimensional dynamics, dimension reduction, source separation and denoising, data summarization, and clustering based on sequence information. The main contributions of this work are:
\begin{enumerate}[leftmargin=1.5em, itemsep=-0.5ex, topsep=0.5ex]
\item We derive kernel transfer operators (KTOs) and empirical estimators (Sections~\ref{ssec:Perron--Frobenius} and~\ref{ssec:Koopman}).  This includes operators on densities in the RKHS rather than mean-embedded measures.

\item We show that the embedded Perron--Frobenius operator is equivalent to the \emph{conditional mean embedding} (CME) formulation (Section~\ref{ssec:Perron--Frobenius}).

\item We propose an eigendecomposition algorithm for RKHS operators (Section~\ref{sec:eigendecomposition}) and show that existing methods for transfer operators are special cases (Section~\ref{ssec:Relationships with other methods}).

\item Lastly, we demonstrate the use of the KTOs in molecular dynamics as well as video, text, and  EEG data analysis (Section~\ref{sec:experiments}).
\end{enumerate}

The remainder of this paper is organized as follows: In Section~\ref{sec:Preliminaries}, we introduce reproducing kernel Hilbert spaces and transfer operators, followed by the kernel formulation of transfer operators in Section~\ref{sec:kernel-transfer-operators}. We demonstrate the proposed methods for the approximation of these operators in Section~\ref{sec:experiments} using several illustrative and real-world examples and conclude with a short summary and future work in Section~\ref{sec:conclusion}.

\section{Notation and Preliminaries}
\label{sec:Preliminaries}

In this section, we discuss various preliminary results necessary for the definition and analysis of kernel transfer operators. The notation and symbols used throughout the manuscript are summarized in Table~\ref{tab:Notation}. For the dynamical systems applications considered below, $ X $ corresponds to the state of the system at time $ t $ and $ Y $ to the state of the system at time $ t + \tau $, where $ \tau $ is a fixed lag time. This will be discussed in detail in Section~\ref{ssec:Transfer Operators}.

\subsection{Reproducing Kernel Hilbert Spaces}

We will first introduce reproducing kernel Hilbert spaces as well as Hilbert space embeddings of probability distributions. See, e.g., \citet{Schoe01, Berlinet04:RKHS,Steinwart2008:SVM} for further details.

\begin{table}[tb]
    \centering
    \caption{The notation and symbols.}
    \begin{tabular}{l@{\hspace{3em}}c@{\hspace{3em}}c}
        \hline
        Random variable & $ X $                     & $ Y $               \\
        Domain          & $ \inspace $              & $ \outspace $       \\
        Observation     & $ x $                     & $ y $               \\
        Kernel function & $ k(x, x^\prime) $        & $ l(y, y^\prime) $  \\
        Feature map     & $ \phi(x) $               & $ \psi(y) $         \\
        Feature matrix  & $ \Phi = [\phi(x_1), \dots, \phi(x_n)] $ & $ \Psi = [\psi(y_1), \dots, \psi(y_n)] $ \\
        Gram matrix     & $ \gram[XX] = \Phi^\top \Phi $           & $ \gram[YY] = \Psi^\top \Psi $           \\
        RKHS            & $ \mathbb{H} $            & $ \mathbb{G} $      \\
        \hline
    \end{tabular}
    \label{tab:Notation}
\end{table}

\begin{definition}[Reproducing kernel Hilbert space, \citep{Schoe01}] \label{def:RKHS}
Let $ \inspace $ be a set and $ \mathbb{H} $ a space of functions $ f \colon \inspace \to \R $. Then $ \mathbb{H} $ is called a \emph{reproducing kernel Hilbert space (RKHS)} with corresponding scalar product $ \innerprod{\cdot}{\cdot}_\mathbb{H} $ and induced norm $ \norm{f}_\mathbb{H} = \innerprod{f}{f}_\mathbb{H}^{1/2} $ if there is a function $ k \colon \inspace \times \inspace \to \R $ such that
\begin{enumerate}[label=(\roman*)]
\item $ \innerprod{f}{k(x, \cdot)}_\mathbb{H} = f(x) $ for all $ f \in \mathbb{H} $ and
\item $ \mathbb{H} = \overline{\mspan\{k(x, \cdot) \mid x \in \inspace \}} $.
\end{enumerate}
\end{definition}

The function $k$ is called a \emph{reproducing kernel} of $\mathbb{H}$. The first requirement, which is called a \emph{reproducing property} of $\mathbb{H}$, in particular implies $ \innerprod{k(x, \cdot)}{k(x^\prime, \cdot)}_\mathbb{H} = k(x, x^\prime) $ for all $ x, x' \in \inspace $. As a result, the function evaluation of $f$ at a given point $ x $ can be regarded as an inner product evaluation in $\mathbb{H}$ between the representer $k(x,\cdot)$ of $x$ and the function itself. Furthermore, we may treat $ k(x, \cdot) $ as a feature map $ \phi(x) $ of $x$ in $\mathbb{H}$ such that $ k(x, x^\prime) = \innerprod{\phi(x)}{\phi(x^\prime)}_\mathbb{H} $. 
Hence, the reproducing kernel $k$ is a kernel in a usual sense with $k(x,\cdot)$ as a \emph{canonical feature map} of $ x $. 
A function $k \colon \inspace\times\inspace\to\R$ is a reproducing kernel (with the aforementioned properties of $\mathbb{H}$) if and only if it is symmetric and \emph{positive definite}, i.e., $k(x,y) = k(y,x)$ and $\sum_{i,j=1}^n c_ic_j k(x_i,x_j) \geq 0$ for any $n\in\mathbb{N}$, any choice of $x_1,\ldots,x_n\in\inspace$, and any $c_1,\ldots,c_n\in\R$ (see \citealt[Chapter 4]{Steinwart2008:SVM}). For example, one of the most commonly used kernels is the Gaussian RBF kernel $k(x,y) = \exp(-\|x-y\|_2^2/2\sigma^2)$ for $x,y\in\inspace$ where $\sigma$ is a bandwidth parameter. More examples of kernels can be found in \citet{Schoe01,Berlinet04:RKHS,Hofmann2008}, and \citet{MFSS16}, for instance. We give an example of a well-known polynomial kernel below.

\begin{example} \label{ex:Kernel}
Let $ \inspace \subset \R^2 $. Consider the polynomial kernel $ k(x, x^\prime) = (1 + \innerprod{x}{x^\prime})^2 $. We could either use the \emph{canonical feature map} $ \phi_\text{can}(x) = k(x,\cdot) $ and the standard RKHS inner product satisfying the reproducing property, the features are then a subset of the function space $\mathbb{H}$, or the \emph{explicit feature map} $ \phi_\text{exp}(x) = [1, \,\sqrt{2}\ts x_1, \,\sqrt{2}\ts x_2, \,x_1^2,\, \sqrt{2} \ts x_1 \ts x_2, \, x_2^2]^\top $ with the standard Euclidean inner product, the features are then a subset of $ \R^6 $. \exampleSymbol
\end{example}

In most applications of kernels, all we need is the inner product between $\phi(x)$ and $\phi(x')$ in $\mathbb{H}$. The kernel trick allows us to evaluate it directly without constructing $\phi$ explicitly. In fact, some kernels such as the Gaussian kernel correspond to infinite-dimensional feature spaces. Most kernel-based learning algorithms rely on computations involving only Gram matrices evaluated on a finite number of data points. That is, the Gram matrix $G \in \R^{n\times n}$ on a data set $x_1,\ldots,x_n$ is given by $G_{ij} = k(x_i,x_j)$. As we will see later, although our transfer operators are defined in terms of $\phi$ and may live in an infinite-dimensional space, all associated operations can be carried out in terms of the finite-dimensional Gram matrices obtained from training data.

%%%%%
\subsection{Hilbert Space Embedding of Distributions}
\label{sec:marginal-mean-embedding}

We can extend the idea of feature maps defined by the kernel function to the space of probability distributions \citep{Berlinet04:RKHS,Smola07Hilbert,MFSS16}. A \emph{kernel mean embedding} provides a feature representation of distributions in RKHS associated with the kernel function.

\begin{definition}[Mean embedding] \label{def:marginal embedding}
Let $ \mathbb{M}_+^1(\inspace) $ be the space of all probability measures $ \pp{P} $ on $ \inspace $ and $\mathbb{H}$ an RKHS endowed with a measurable real-valued kernel $k \colon \inspace \times \inspace \rightarrow \R $ such that $\sup_{x\in\inspace} k(x,x) < \infty$. Then the \emph{kernel mean embedding} $ \mu_\pp{P} \in \mathbb{H} $ is defined by
\begin{equation*}
    \mu_\pp{P}
        := \mathbb{E}_{\scriptscriptstyle X}[\phi(X)]
        = \int \phi(x) \ts \dd \pp{P}(x)
        = \int k(x, \cdot) \ts \dd \pp{P}(x),
\end{equation*}
where $\mu_{\pp{P}}$ is a Bochner integral (see, e.g., \citet[Chapter 2]{Diestel-77} and \citet[Chapter 1]{Dinculeanu:2000} for the definition of the Bochner integral).
\end{definition}

In practice, we only have access to a $ \pp{P}(X) $-distributed  sample set $ \mathbb{D}_{\scriptscriptstyle X} = \{ x_1, \dots, x_n \} $. The empirical estimate of $\mu_{\pp{P}}$ can be computed as
\begin{equation*}
     \hat{\mu}_\pp{P}
        = \frac{1}{n} \sum_{i=1}^n \phi(x_i)
        = \frac{1}{n} \sum_{i=1}^n k(x_i, \cdot)
        = \frac{1}{n} \Phi \mathds{1},
\end{equation*}
where $ \Phi = [\phi(x_1), \dots, \phi(x_n)] $ is the feature matrix\footnote{Although this term is commonly used in the literature, $ \Phi $ is technically not a matrix, but a row vector in $ \mathbb{H}^n $. If $\mathbb{H}$ is finite-dimensional, $ \Phi $ can be viewed as a matrix.} and $ \mathds{1} = [1,\,\dots,\,1]^\top $ the vector of ones. By the reproducing property of $\mathbb{H}$, we have $\mathbb{E}_{X\sim\pp{P}}[f(X)] = \langle f,\mu_{\pp{P}} \rangle_{\mathbb{H}}$ and $\widehat{\mathbb{E}}_{X\sim\pp{P}}[f(X)] = \langle f,\hat{\mu}_{\pp{P}} \rangle_{\mathbb{H}}$ for all $f\in\mathbb{H}$.

Different choices of kernel functions result in different representations of the distribution~$\pp{P}$. In particular, the kernel mean embedding $\mu_{\pp{P}}$ fully characterizes $\pp{P}$ if $k$ is a \emph{characteristic kernel} \citep{Fukumizu04,Sriperumbudur08injectivehilbert}. In other words, we do not lose any information about $\pp{P}$ by embedding it into a characteristic RKHS. Examples of characteristic kernels include the Gaussian RBF kernel defined above and the Laplacian kernel $k(x,y) = \exp(-\|x-y\|_2/\sigma)$.

\begin{definition}[Integral operator]
\label{def:Integral operator}
Let $\inspace$ be a compact Hausdorff space, $\nu$ a finite Borel measure with support $\inspace$, and $k$ a continuous positive definite kernel on $\inspace$. An \emph{integral operator} $\ebd[k] \colon L_2(\inspace,\nu)\to L_2(\inspace,\nu)$ is defined by
\begin{equation*}
    (\ebd[k] f)(\cdot) := \int_{\inspace} k(x, \cdot) f(x) \ts \dd \nu(x).
\end{equation*}
\end{definition}  

In what follows, we consider $\nu$ to be the Lebesgue measure unless it is stated otherwise. We will sometimes omit the subscript if it is clear which kernel is meant. It was shown in \citet{Kato80} that if $ \int_\inspace \abs{k(x, y)} \dd x \le M_1 $, $ \int_\inspace \abs{k(x, y)} \dd y \le M_2 $, and $ f \in L^r(\inspace) $, with $ 1 \le r \le \infty $, then we obtain $ \norm{\ebd[k] f} \le \max(M_1, M_2) \norm{f} $ and the operator is bounded. Here, $ L^r(\inspace) $, with $ 1 \le r \le \infty $, denotes the spaces of $ r $-Lebesgue integrable functions. Since in our case $ k $ is symmetric, we obtain $ M_1 = M_2 $. In particular, if $ \inspace $ is compact and $ k(x, y) $ continuous in $ x $ and $ y $, this is satisfied. Whenever $ \pp{P} $ has a density $ p $, this means $ \mu_\pp{P} = \ebd[k] p $.

We can also generalize the idea of mean embedding of marginal distributions $\mathbb{P}(X)$ to conditional distributions $\mathbb{P}(Y\,|\,X)$. To this end, we first need to introduce the concept of covariance operators in Hilbert spaces~\citep{Baker70:XCov, Baker1973}.

\begin{definition}[Covariance operators]
Let $(X,Y)$ be a random variable on $ \inspace \times \outspace $ with corresponding marginal distributions $\pp{P}(X)$ and $\pp{P}(Y)$, respectively, and joint distribution $\pp{P}(X,Y)$. Let $\phi$ and $\psi$ be feature maps associated with the bounded kernels $k$ and $l$, respectively.\!\footnote{For specific choices of $\pp{P}(X)$ and $\pp{P}(Y)$, the boundedness assumption can be replaced by a more general integrability assumption, i.e., $\mathbb{E}_{\scriptscriptstyle X}[k(X,X)] < \infty$ and $\mathbb{E}_{\scriptscriptstyle Y}[l(Y,Y)] < \infty$, so that $\mathbb{H} \subset L^2(\inspace, \pp{P}(X))$ and $\mathbb{\mathbb{G}} \subset L^2(\outspace, \pp{P}(Y))$, respectively.} Let $\mathbb{H}$ and $\mathbb{G}$ be the RKHSs associated with the kernels $k$ and $l$, respectively. Then the \emph{covariance operator} $ \cov[XX] \colon \mathbb{H} \to \mathbb{H} $ and the \emph{cross-covariance operator} $ \cov[YX] \colon \mathbb{H} \to \mathbb{G} $ are defined as
\begin{alignat*}{4}
    \cov[XX] & := \int \phi(X) \otimes \phi(X) \ts \dd \pp{P}(X)
             &&= \mathbb{E}_{\scriptscriptstyle X}[\phi(X) \otimes \phi(X)], \\
    \cov[YX] & := \int \psi(Y) \otimes \phi(X) \ts \dd \pp{P}(Y,X)   
             &&= \mathbb{E}_{\scriptscriptstyle \mathit{YX}}[\psi(Y) \otimes \phi(X)].
\end{alignat*}
\end{definition}
 
\begin{remark}
Note that $ \psi(y) \otimes \phi(x) $ defines a rank-one operator from $ \mathbb{H} $ to $ \mathbb{G} $ via
\begin{equation*}
    \big(\psi(y) \otimes \phi(x)\big) f = \innerprod{\phi(x)}{f}_\mathbb{H} \psi(y) = f(x) \ts \psi(y)
\end{equation*}
so that $ \innerprod{\big(\psi(y) \otimes \phi(x)\big) f}{g}_\mathbb{G} = f(x) \innerprod{\psi(y)}{g}_\mathbb{G} = f(x) \ts g(y) $.
\end{remark}

The centered counterparts of $ \cov[XX] $ and $\cov[YX]$ are defined similarly using the mean-sub\-tracted feature maps $\phi_c(X) = \phi(X) - \mu_{\pp{P}(X)}$ and $\psi_c(Y) = \psi(Y) - \mu_{\pp{P}(Y)}$, where $\mu_{\pp{P}(X)} := \mathbb{E}_{\scriptscriptstyle X}[\phi(X)]$ and $\mu_{\pp{P}(Y)} := \mathbb{E}_{\scriptscriptstyle Y}[\psi(Y)]$. Intuitively, one may think of $ \cov[XX] $ and $\cov[YX]$ as a nonlinear generalization of covariance and cross-covariance matrices. We can express the cross-covariance of two functions $ f \in \mathbb{H} $ and $ g \in \mathbb{G} $ in terms of $ \cov[XY] $ and $ \cov[YX] $ as
\begin{equation} \label{eq:Cross-covariance}
    \mathbb{E}_{\scriptscriptstyle XY}[f(X) \ts g(Y)]
        = \innerprod{f}{\cov[XY] g}_\mathbb{H}
        = \innerprod{\cov[YX] f}{g}_\mathbb{G}.
\end{equation}
Hence, $\cov[XY]$ is the adjoint of $\cov[YX]$ and $\cov[XX]$ is a self-adjoint operator.

The following result, which relates $\cov[XX]$ and $\cov[XY]$, will be used to define the embedding of conditional distributions. We refer the readers to \citet{Fukumizu04} for the proof.

\begin{proposition} \label{thm:Fukumizu04}
If $ \ts \mathbb{E}_{\scriptscriptstyle Y \mid X}[g(Y) \mid X = \cdot \,] \in \mathbb{H} $ for all $ g \in \mathbb{G} $, $\cov[XX] \mathbb{E}_{\scriptscriptstyle Y \mid X}[g(Y) \mid X = \cdot \,] = \cov[XY] g$.
\end{proposition}

In general, the covariance operator and cross-covariance operator cannot be computed directly since the joint distribution $ \pp{P}(X, Y) $ is typically not known. We can, however, estimate it from sampled data. Given $ n $ pairs of training data $ \mathbb{D}_{\scriptscriptstyle XY} = \{(x_1, y_1), \dots, (x_n, y_n)\} $ drawn i.i.d.\ from the probability distribution $ \pp{P}(X, Y) $, we define the feature matrices
\begin{equation*}
    \Phi = \begin{bmatrix} \phi(x_1) & \dots & \phi(x_n) \end{bmatrix}
    \quad \text{and} \quad
    \Psi = \begin{bmatrix} \psi(y_1) & \dots & \psi(y_n) \end{bmatrix}.
\end{equation*}
The corresponding Gram matrices are given by $ \gram[XX] = \Phi^\top \Phi $ and $ \gram[YY] = \Psi^\top \Psi $. Then, the empirical estimates of $ \cov[XX] $ and $\cov[YX]$ are given by
\begin{equation*}
    \ecov[XX] = \frac{1}{n} \sum_{i=1}^n \phi(x_i) \otimes \phi(x_i)
              = \frac{1}{n} \Phi \Phi^\top
    \quad \text{and} \quad
    \ecov[YX] = \frac{1}{n} \sum_{i=1}^n \psi(y_i)\otimes\phi(x_i)
              = \frac{1}{n} \Psi \Phi^\top.
\end{equation*}
Analogously, the mean-subtracted counterparts of $\ecov[XX]$ and $\ecov[YX]$ can be obtained as $ \frac{1}{n} \Phi H \Phi^\top $ and $ \frac{1}{n} \Psi H \Phi^\top $, where $ H $ is the centering matrix given by $ H = \id_n - \frac{1}{n} \mathds{1}_n \mathds{1}_n^\top $. Note that if both $ k $ and $ l $ are linear kernels for which $\phi$ and $\psi$ are identity maps, we obtain covariance and cross-covariance matrices as a special case.

%\subsection{Hilbert Space Embedding of Conditional Distributions}
%\label{sec:conditional-mean-embedding}

We are now in a position to introduce the Hilbert space embedding of conditional distributions, commonly known as \emph{conditional mean embedding}. Interested readers should consult \citet{SHSF09, Song2013, MFSS16} for further details on this topic.

\begin{definition}[Conditional mean embedding, \citealt{SHSF09}] \label{def:cme}
Let $ \cov[XX] $ be the covariance operator on $\mathbb{H}$ and $ \cov[YX] $  be the cross-covariance operator from $\mathbb{H}$ to $\mathbb{G}$, respectively. Assume that Proposition \ref{thm:Fukumizu04} holds. Then the conditional mean embedding of $\pp{P}(Y \mid X)$ is an operator mapping from $\mathbb{H}$ to $\mathbb{G}$ and is given by $ \cme := \cov[YX] \ts \cov[XX]^{-1} $.
\end{definition}

\noindent Under the assumption that $ \mathbb{E}_{\scriptscriptstyle Y \mid X}[g(Y) \mid X = \cdot \,] \in \mathbb{H} $ for all $ g \in \mathbb{G} $, it follows from the reproducing property of $\mathbb{H}$ and Proposition \ref{thm:Fukumizu04} that
\begin{equation*}
    \mathbb{E}_{\scriptscriptstyle Y \mid x}[g(Y) \mid X = x]
        = \innerprod{\cov[XX]^{-1} \ts \cov[XY] \ts g}{k(x,\cdot)}_{\mathbb{H}}
        = \innerprod{g}{\cov[YX] \ts \cov[XX]^{-1} \ts k(x,\cdot)}_{\mathbb{G}}
\end{equation*}
for all $ g \in \mathbb{G} $. That is, we can treat $\mu_{\scriptscriptstyle Y \mid x} := \cme k(x,\cdot) = \cov[YX] \ts \cov[XX]^{-1} \ts k(x, \cdot)$ as the conditional mean embedding of $\pp{P}(Y \mid X = x)$ in $\mathbb{G}$. 

\begin{remark}
As noted by \citet{Fukumizu13:KBR}, the assumption that $ \/ \mathbb{E}_{\scriptscriptstyle Y \mid X}[g(Y) \mid X = \cdot \,] \in \mathbb{H} $ for all $ g \in \mathbb{G} $ may not hold in general. Hence, the expression $\cov[YX] \ts \cov[XX]^{-1} \ts k(x, \cdot)$ is used as an approximation of the conditional mean $\mu_{\scriptscriptstyle Y \mid x}$. A common approach to alleviate this problem is to consider the regularized inverse $(\cov[XX] + \varepsilon \ts \idop)^{-1}$, where $\varepsilon > 0$ is a regularization parameter and $ \idop $ is the identity operator in $\mathbb{H}$. The empirical estimator of the conditional mean embedding is then given by
\begin{equation*}
    \ecme = \ecov[YX] (\ecov[XX] + \varepsilon \ts \idop)^{-1}
          = \Psi (\gram[XX] + n \ts \varepsilon \ts \id_n)^{-1} \Phi^\top.
\end{equation*}
\citet{Fukumizu13:KBR} establishes the consistency and convergence rate of this estimator under appropriate assumptions. Throughout this work, we consider $\cme := \cov[YX] \ts (\cov[XX] + \varepsilon\mathcal{I})^{-1} $.
\end{remark}

\subsection{Transfer Operators}
\label{ssec:Transfer Operators}

We now give a brief introduction to transfer operators and their applications. A detailed exposition on this topic can be found in~\citet{KNKWKSN18}. Conditions under which these operators exist are described in \citet{LaMa94}, summarized also in \citet{KKS16}. In what follows, we assume that the operators are compact and have a discrete spectrum, see \citet{Sch99, Huisinga01} for details. Let $ \{ X_t \}_{t \ge 0} $ be a stationary and ergodic Markov process defined on the state space $ \inspace \subset \R^d $. Then the \emph{transition density function} $ p_\tau $ is defined by 
\begin{equation*}
    \pp{P}[X_{t + \tau} \in \mathbb{A} \mid X_t = x] = \intop_{\mathbb{A}} p_\tau(y \mid x) \, \dd y,
\end{equation*} where $ \mathbb{A} $ is any measurable set. That is, $ p_\tau(y \mid x) $ is the conditional probability density of $ X_{t + \tau} = y $ given that $ X_t = x $, and by the stationarity and Markov assumptions this captures all information about the stochastic process at discretization level $\tau$.

\begin{definition}[Transfer operators] \label{def:Transfer operators}
Let $ p_t \in L^1(\inspace) $ be a probability density and $ f_t \in L^\infty(\inspace)$ an observable\footnote{An observable could be, for example, a measurement or sensor probe.} of the system. For a given lag time $ \tau $:
\begin{enumerate}[label=(\roman*), itemsep=0ex, topsep=1ex]
\item The \emph{Perron--Frobenius operator} $ \pf \colon L^1(\inspace) \to L^1(\inspace) $ is defined by
\begin{equation*}
    \left(\pf p_t \right)(y) = \int p_\tau(y \mid x) \ts p_t(x) \ts \dd x.
\end{equation*}
\item The \emph{Koopman operator} $\ko \colon L^{\infty}(\inspace) \to L^{\infty}(\inspace) $
is defined by
\begin{equation*}
   \left(\ko f_t\right)(x)
        = \int p_\tau(y \mid x) \ts f_t(y) \ts \dd y
        = \mathbb{E}[f_t(X_{t + \tau}) \mid X_t = x].
\end{equation*}
\end{enumerate}
\end{definition}

Note that the operators and the corresponding eigenvalues implicitly depend on the lag time $ \tau $. A density $ \pi $ that is invariant under the action of $ \pf $ is called \emph{invariant, equilibrium} or \emph{stationary density}. That is, it holds that $ \pf \pi = \pi $. For the following definition, we assume that there is a unique invariant density $ \pi > 0 $, which for molecular dynamics problems is given by the Boltzmann distribution $ \pi \sim \exp(-\beta V)$, where $ \beta $ is the inverse temperature and $ V $ the potential of the system \citep{SS13}.

\begin{definition}[Transfer operators cont'd]
Let $ u_t(x) = \pi(x)^{-1} \ts p_t(x)$ be a probability density with respect to the equilibrium density $ \pi $.
\begin{enumerate}[label=(\roman*), itemsep=0ex, topsep=1ex] \setcounter{enumi}{2}
\item The \emph{Perron--Frobenius operator with respect to the equilibrium density}, denoted by~$ \mathcal{T} $, is defined as
\begin{equation*}
   \left(\mathcal{T} u_t\right)(y) = \frac{1}{\pi(y)} \int p_\tau(y \mid x) \ts \pi(x) \ts u_t(x) \ts \dd x.
\end{equation*}
\end{enumerate}
\end{definition}

Under certain conditions, the transfer operators can be defined on other spaces~$ L^r $ and $ L^{r'} $, with~$ r \ne 1 $ and~$ r' \ne \infty $, see \citet{BaRo95, KKS16}. The operators $ \pf $ and $\ko $ are adjoint to each other with respect to $ \innerprod{\cdot}{\cdot} $, defined by $ \innerprod{f}{g} = \int_{\inspace}f(x) \ts g(x) \ts \dd x $, while $ \mathcal{T} $ and $ \ko $ are adjoint with respect to $ \innerprod{\cdot}{\cdot}_\pi $, defined by $ \innerprod{f}{g}_\pi = \int_\inspace f(x) \ts g(x) \ts \pi(x) \ts \dd x$ for $f\in L_\pi^r(\inspace)$ and $g\in L_\pi^{r'}(\inspace)$ where $\frac{1}{r} + \frac{1}{r'} = 1$. That is, we have $\innerprod{\ko f}{g}_\pi = \innerprod{f}{\mathcal{T} g}_\pi$.

\begin{definition}[Reversibility]
A system is called reversible if the \emph{detailed balance condition} $ \pi(x) \ts p_\tau(y \mid x) = \pi(y) \ts p_\tau(x \mid y) $ holds for all $ x, y \in \inspace $.
\end{definition}

If the system is reversible, then $ \ko = \mathcal{T} $. Moreover, the operators' eigenvalues $ \lambda_\ell $ are real and the eigenfunctions $ \varphi_{\ell} $ form an orthogonal basis with respect to the corresponding scalar product. As a result, the eigenvalues can be sorted in descending order so that $ 1 = \lambda_1 > \lambda_2 \ge \lambda_3 \ge \dots $. See~\citet{NoNu13, NKPMN14, KNKWKSN18} for more details.

In what follows, we will use the Ornstein--Uhlenbeck process, which is reversible and whose eigenvalues and eigenfunctions can be computed analytically, as a guiding example.

\begin{example} \label{ex:Ornstein-Uhlenbeck}
A one-dimensional Ornstein--Uhlen\-beck process is defined by the stochastic differential equation $ \dd X_t = -\alpha D X_{t} \ts \dd t + \sqrt{2D} \ts \dd W_t $, where $ \alpha $ is the friction coefficient, $ D = \beta^{-1} $ the diffusion coefficient, and $ \{W_{t}\}_{t \ge 0} $ a one-dimensional standard Wiener process.
The solution is given by
\begin{equation*}
    X_t = X_0 \ts e^{-\alpha D t} + \sqrt{2D} \int_0^t e^{-\alpha D(t-s)} \dd W_s,
\end{equation*}
which is a Gaussian process with time marginal $ X_t \sim \mathcal{N}\big(X_0 \ts e^{-\alpha D t}, \frac{1}{\alpha} (1 - e^{-2 \alpha D t})\big) $.
\exampleSymbol
\end{example}

\begin{remark}
Definition~\ref{def:Transfer operators} introduces the stochastic Koopman operator. For a deterministic dynamical system of the form $ \dot{x} = F(x) $, we obtain $ p_\tau(y \mid x) = \delta_{\boldsymbol{\Phi}_\tau(x)}(y) $, where $ \boldsymbol{\Phi}_\tau $ denotes the flow map mapping $x(t)$ to $x(t+\tau)$ and $ \delta_x $ the Dirac distribution centered in $ x $. Thus, $\ko f = f \circ \boldsymbol{\Phi}_\tau $. For a discrete dynamical system of the form $ x_{i+1} = F(x_i) $, we obtain $ \mathcal{K} f = f \circ F $. In the same way, the Perron--Frobenius operator can be defined for deterministic systems, see, e.g., \citet{LaMa94, KKS16}.
\end{remark}

\begin{example} \label{ex:Simple system}
Consider the discrete dynamical system $ F \colon \R^2 \to \R^2 $, taken from \citet{TRLBK14}, with
\begin{equation*}
    \begin{bmatrix}
        x_1 \\
        x_2
    \end{bmatrix}
    \mapsto
    \begin{bmatrix}
        a \ts x_1 \\
        b \ts x_2 + (b - a^2) \ts x_1^2 \ts
    \end{bmatrix}.
\end{equation*}
For the numerical experiments, we set $ a = 0.8 $ and $ b = 0.7 $. The eigenvalues of the Koopman operator associated with the system are $ \lambda_1 = 1 $, $ \lambda_2 = a $, and $ \lambda_3 = b $ with corresponding eigenfunctions $ \varphi_1(x) = 1 $, $ \varphi_2(x) = x_1 $, and $ \varphi_3(x) = x_2 + x_1^2 $. Furthermore, products of eigenfunctions are again eigenfunctions, for instance, $ \varphi_4(x) = \varphi_2(x)^2 = x_1^2$ with eigenvalue $ \lambda_4 = \lambda_2^2 = a^2 $. \exampleSymbol
\end{example}

Given the eigenvalues and eigenfunctions of the Koopman operator, we can predict the evolution of the dynamical system. To this end, let $ g(x) = x $ be the \emph{full-state observable}. We then write $ g(x) $ in terms of the eigenfunctions as $ g(x) = x = \sum_{\ell} \varphi_{\ell}(x) \ts \eta_{\ell} $ where the vectors $ \eta_\ell $ are called \emph{Koopman modes}. Defining the Koopman operator to act componentwise for vector-valued functions, we obtain
\begin{equation*}
    \left(\ko g\right)(x)
        = \mathbb{E}[g(X_{\tau}) \mid X_{0}=x]
        = \sum_{\ell}\lambda_{\ell}(\tau) \ts \varphi_{\ell}(x) \ts \eta_{\ell}.
\end{equation*}

\begin{example}
For the simple deterministic system introduced in Example~\ref{ex:Simple system}, we obtain the Koopman modes $ \eta_1 = [0, \, 0]^\top $, $ \eta_2 = [1, \, 0]^\top $, $ \eta_3 = [0, \, 1]^\top $, and $ \eta_4 = [0, \, -1]^\top $ so that $ g(x) = \sum_{\ell=1}^4 \varphi_\ell(x) \ts \eta_\ell $ and $ \mathcal{K} g(x) = \sum_{\ell=1}^4 \lambda_\ell \ts \varphi_\ell(x) \ts \eta_\ell = F(x) $.
\end{example}

The above example illustrates that with the aid of the Koopman decomposition into eigenvalues, eigenfunctions, and modes, we can now evaluate the dynamical system at any data point. This is particularly useful if the system is not known explicitly.

\section{Kernel Transfer Operators and Their Eigendecompositions}
\label{sec:kernel-transfer-operators}

We now express the transfer operators introduced above in terms of the covariance and cross-covariance operators defined on some RKHS $\mathbb{H}$. Since $X_t$ and $X_{t+\tau}$ always live in the same state space, we assume that the input and output spaces and thus also the kernels and resulting Hilbert spaces are identical, i.e., $\inspace = \outspace$, $k = l$, and $\mathbb{H} = \mathbb{G}$. In addition to the standard transfer operators, we will derive transfer operators for embedded densities and observables in the RKHS $\mathbb{H}$ and analyze the relationships between them. To this end, we define---similar to the standard Gram matrices $ \gram[XX] $ and $ \gram[YY] $---the \emph{time-lagged Gram matrices} $ \gram[XY] = \Phi^\top \Psi $ and $ \gram[YX] = \Psi^\top \Phi $. For the RKHS community, our result on the kernel Perron--Frobenius operator is of particular interest, as it describes propagation of densities directly in the RKHS, unlike the conditional mean embedding which does so only in embedded form.

\subsection{Main Results}

Overall, we derive four different kernel transfer operators that can be written in terms of the covariance and cross-covariance operators as summarized in Table~\ref{tab:RKHS transfer operators}. The kernel Perron--Frobenius operator maps densities $ p_t $ to densities $ p_{t + \tau} $, the kernel Koopman operator maps an observable function $ f $ to its expected value function $ \mathbb{E}[f(X_{t + \tau}) \mid X_t = \cdot\,] $ under the assumption that the densities and observables (and the densities and observables pushed forward) are in $ \mathbb{H} $. This basically means that we assume that the RKHS $ \mathbb{H} $ is an invariant subspace of the respective operator, which is a strong assumption. Depending on the kernel, the feature space might be low-dimensional (e.g., polynomial kernel), but could also be infinite-dimensional (e.g., Gaussian kernel). While the invariance might be satisfied for simple drift-diffusion processes, it will in general not be satisfied for more complex chaotic systems, and the class of dynamical systems for which this holds is (to our knowledge) unknown. The embedded operators do not assume densities or observables to be in $ \mathbb{H} $ and rather push forward kernel mean embeddings (see Definition~\ref{def:marginal embedding}) or embedded observable functions (see Definition~\ref{def:Integral operator}), respectively. We prove that these estimates converge at a rate of order $O_p\left(n^{-1/2}\varepsilon^{-1}\right)$ in operator norm, where $\varepsilon$ is a regularization parameter.

\begin{table}[tbh]
    \centering
    \caption{Overview of kernel transfer operators where $A := \gram[XY]^{-1} \ts (\gram[XX] + n \varepsilon \id)^{-1} \ts \gram[XY] $.}
    \setlength\tabcolsep{2ex}
    \renewcommand{\arraystretch}{1.5}
    \begin{tabular}{lll}
        \hline
        & \multicolumn{1}{c}{Perron--Frobenius}
        & \multicolumn{1}{c}{Koopman} \\ \hline
        Kernel operator
        & $ \begin{array}{r@{\;}c@{\;}l} \pf[k] &=& (\cov[XX] + \varepsilon \idop)^{-1} \cov[YX] \\[-1ex]
        &\approx & \Psi \ts A \ts \Phi^\top \end{array} $
        & $ \begin{array}{r@{\;}c@{\;}l} \ko[k] &=& (\cov[XX] + \varepsilon \idop)^{-1} \cov[XY] \\[-1ex]
        &\approx& \Phi \ts (\gram[XX] + n \varepsilon \id)^{-1} \ts \Psi^\top \end{array} $ \\ \hline
        Embedded operator
        & $ \begin{array}{r@{\;}c@{\;}l} \pf[\ebd] &=& \cov[YX] (\cov[XX] + \varepsilon \idop)^{-1} \\[-1ex]
        &\approx&  \Psi \ts (\gram[XX] + n \varepsilon \id)^{-1} \ts \Phi^\top \end{array} $
        & $ \begin{array}{r@{\;}c@{\;}l} \ko[\ebd] &=& \cov[XY] (\cov[XX] + \varepsilon \idop)^{-1} \\[-1ex]
        &\approx& \Phi \ts A^\top \ts \Psi^\top \end{array} $ \\
        \hline
    \end{tabular}
    \label{tab:RKHS transfer operators}
\end{table}

Moreover, we provide a method for computing the eigendecomposition of general finite--rank RKHS operators by solving a real-valued surrogate problem based on the Gram matrices. While this result is important by itself, it results in the eigendecompositions given in Table~\ref{tab:TO eigendecomposition} for the kernel transfer operators derived in this paper. The detailed derivations of our main results can be found in the sections below.

\begin{table}[htb]
    \centering
    \caption{Surrogate real eigenproblem and resulting operator eigenfunction.}
    \setlength\tabcolsep{2ex}
    \renewcommand{\arraystretch}{1.5}
    \begin{tabular}{lll}
        \hline
        & \multicolumn{1}{c}{Perron--Frobenius} & \multicolumn{1}{c}{Koopman} \\
        \hline
        Kernel operator & $(\gram[XX]+n\varepsilon \id)^{-1}\gram[XY] \ts \mathbf{v} = \lambda \ts \mathbf{v}$ & $ (\gram[XX] + n \varepsilon \id)^{-1}\gram[YX] \ts \mathbf{v} = \lambda \ts \mathbf{v}$ \\
        Embedded operator & $\gram[XY] (\gram[XX] + n \varepsilon \id)^{-1} \ts \mathbf{v} = \lambda \ts \mathbf{v}$ & $ \gram[YX] (\gram[XX] + n \varepsilon \id)^{-1} \ts \mathbf{v} = \lambda \ts \mathbf{v}$ \\
        Eigenfunction & $ \varphi = \Phi \ts \gram[XX]^{-1} \ts \mathbf{v}$ & $ \varphi = \Phi \ts \mathbf{v} $ \\
        \hline
    \end{tabular}
    \label{tab:TO eigendecomposition}
\end{table}

\subsection{Eigendecomposition of RKHS Operators}
\label{sec:eigendecomposition}

If the feature space is finite-dimensional and known explicitly, we can compute eigenfunctions directly as we will show in Example~\ref{ex:EDMD simple system}. The advantage of that approach is that the matrix size does not depend on the number of test points $ n $. As $ n \to \infty $, this approach converges to a Galerkin approximation of the respective operator \citep{KKS16}. Now, we want to consider also the cases where the dimension of the feature space is larger than the number of test points or where the feature space is even infinite-dimensional. Let $ \mathcal{S} = \Upsilon B \ts \Gamma^\top $ be a Hilbert--Schmidt operator mapping from $ \mathbb{H} $ to itself, with $ \Upsilon, \Gamma$ row vectors in $\mathbb{H}^n$, and $ B \in \R^{n \times n} $ for some~$ n $. Assume, furthermore, that $ \Upsilon$ and $ \Gamma$ contain linearly independent elements, which is for instance the case if we use a Gaussian kernel, choose the observations $ z_1, \dots, z_n $ as well as $ z'_1, \dots, z'_n $ to be pairwise different, and define $ \Upsilon = [k(z_1,\cdot), \dots, k(z_n,\cdot)] $, $\Gamma= [k(z'_1,\cdot), \dots, k(z'_n,\cdot)] $. Then the eigenvalues and eigenfunctions of $ \mathcal{S} $ can be computed from eigenvalues and eigenvectors of $ \gram[\Gamma\Upsilon] \ts B$ or $B \ts \gram[\Gamma\Upsilon] $, where $ \gram[\Gamma\Upsilon] = \Gamma^\top \Upsilon$.

\begin{proposition} \label{prop:eigenfunction-explicit}
	The Hilbert--Schmidt operator $ \mathcal{S} = \Upsilon B \Gamma^\top $ has an eigenvalue $\lambda \ne 0$ with corresponding eigenfunction $ \upsilon = \Upsilon \mathbf{v} $ if and only if $ \mathbf{v} $ is an eigenvector of $ B \ts \gram[\Gamma\Upsilon] $ associated with $ \lambda $.
	Similarly, $\mathcal{S} $ has an eigenvalue $\lambda \ne 0$ with corresponding eigenfunction $ \gamma = \Gamma \ts \gram[\Gamma\Gamma]^{-1} \ts \mathbf{v} $ if  and only if $ \mathbf{v}$ is an eigenvector of $ \gram[\Gamma\Upsilon] \ts B $.
\end{proposition}

\begin{proof}
	Let $ \upsilon = \Upsilon \mathbf{v} $ be an eigenfunction of $\mathcal{S}$ associated with $ \lambda $. Then
	\begin{equation*}
	\mathcal{S} \upsilon = \lambda \upsilon 
	\quad \Leftrightarrow \quad
	\Upsilon \ts B \ts \gram[\Gamma\Upsilon] \ts \mathbf{v} = \lambda \ts \Upsilon \ts \mathbf{v}
	\quad \Leftrightarrow \quad
	B \ts \gram[\Gamma\Upsilon] \ts \mathbf{v} = \lambda \ts \mathbf{v}.
	\end{equation*}
	For the second part, let $ \gamma = \Gamma \gram[\Gamma\Gamma]^{-1} \mathbf{v}$ be an eigenfunction of $\mathcal{S}$. Then
	\begin{equation*}
	\mathcal{S} \gamma = \lambda \gamma
	\quad \Leftrightarrow \quad
	\Upsilon B \ts \gram[\Gamma\Gamma] \ts \gram[\Gamma\Gamma]^{-1} \ts \mathbf{v} = \lambda \ts  \Gamma \ts \gram[\Gamma\Gamma]^{-1} \ts \mathbf{v}
	\quad \Leftrightarrow \quad
	\gram[\Gamma\Upsilon] \ts B \ts \mathbf{v} = \lambda \mathbf{v}. \qedhere
	\end{equation*}
\end{proof} 

We thus obtain eigendecomposition expressions for all transfer operator estimators as listed in Table~\ref{tab:TO eigendecomposition}. We will use these for the experiments in Section~\ref{sec:experiments}.

\begin{example}
KPCA \citep{Scholkopf98:KPCA} can be interpreted as an eigendecomposition of the centered covariance operator $\cov[XX] - \mu_{\pp{P}(X)} \otimes \mu_{\pp{P}(X)}$, as already noted in \citet{sejdinovic2014KMH}. Our eigendecomposition result thus includes KPCA as a special case and extends it to RKHS operators of more general form.
\end{example}

\begin{example}
Let us analyze the Ornstein--Uhlenbeck process introduced in Example~\ref{ex:Ornstein-Uhlenbeck}. We use $ \tau = \frac{1}{2} $, $ \alpha = 4 $, and $ D = \frac{1}{4} $ and generate 5000 uniformly distributed test points in $ [-2, 2] $. Furthermore, we use the Gaussian kernel with $ \sigma^2 = 0.3 $. Applying our eigendecomposition result to the kernel Perron--Frobenius operator---since the test points are distributed uniformly, we obtain $ \pf[k] $, see Corollary~\ref{col:PF derivation}---and kernel Koopman operator yields the results shown in Figure~\ref{fig:OU kEDMD}. This special case is equivalent to the kernel EDMD method.\exampleSymbol
	
\begin{figure}[htb]
	\centering 
	\begin{minipage}{0.49\textwidth}
		\centering
		\subfiguretitle{a)}
		\includegraphics[width=0.9\textwidth]{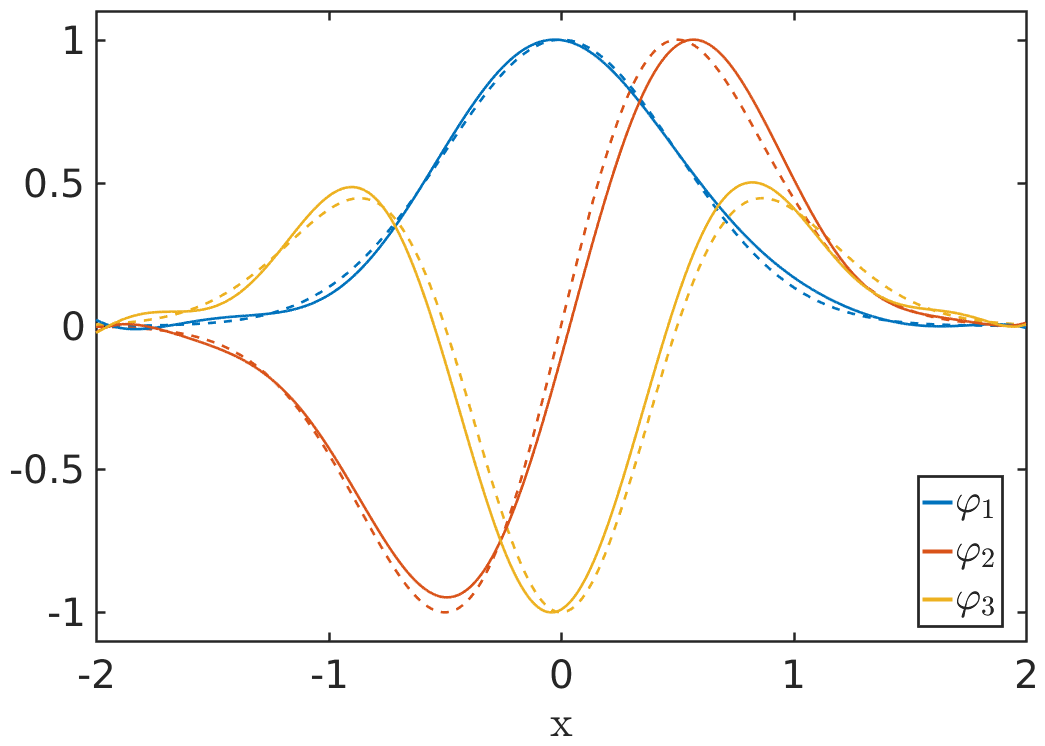}
	\end{minipage}
	\begin{minipage}{0.49\textwidth}
		\centering
		\subfiguretitle{b)}
		\includegraphics[width=0.9\textwidth]{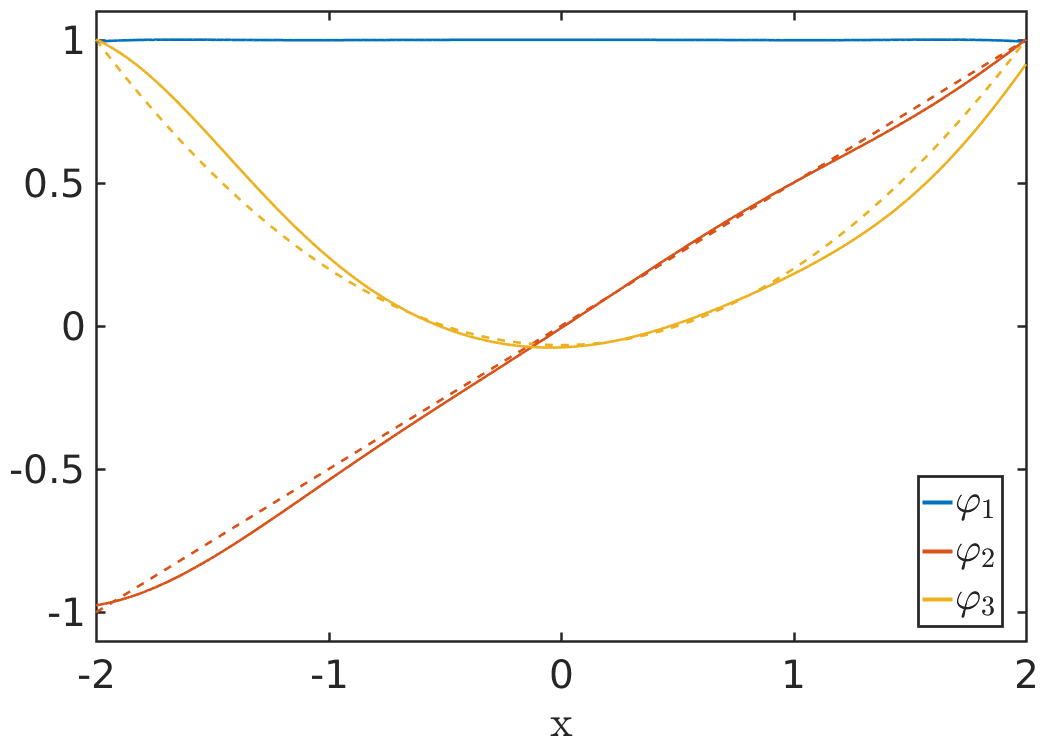}
	\end{minipage}
	\caption{a) Dominant eigenfunctions of the Perron--Frobenius operator $ \pf $ associated with the Ornstein--Uhlenbeck process. b) Dominant eigenfunctions of the Koopman operator $ \ko $. The solid lines are the numerically computed and the dotted lines the analytically computed eigenfunctions.}
	\label{fig:OU kEDMD}
\end{figure}

\end{example}

\subsection{Perron--Frobenius Operators}
\label{ssec:Perron--Frobenius}
In this section, we derive operators pushing forward densities $ p_t \in \mathbb{H}$ to $ p_{t+\tau} \in \mathbb{H} $ or kernel mean embeddings $ \mu_t $  to $ \mu_{t+\tau} $, respectively.

\subsubsection{Kernel Perron--Frobenius Operator}

First we consider the \emph{kernel Perron--Frobenius operator} $ \pf[k] $ defined on the RKHS $ \mathbb{H} $ endowed with the kernel $ k $. That is, we derive the operator pushing $ p_t $ forward to $ p_{t+\tau} $, assuming that $ p_t \in \mathbb{H} $ and $ p_{t+\tau} \in \mathbb{H} $. In general, this will not be the case and the question is how and under which conditions this operator approximates the Perron--Frobenius operator defined on $ L^1 $. Unlike the conditional mean operator \citep{Song2013}, the kernel Perron--Frobenius  operator describes propopagation of densities in the RKHS directly rather than in embedded form. We provide results on the convergence rate of the estimate in Section \ref{sec:consistency}. For kernels with explicitly given feature spaces, related convergence results can be found in~\citet{WKR15, KKS16, KoMe18}. Notice that the assumption that the relevant densities are in $\mathbb{H}$ is different from the embedding approach discussed in Section~\ref{sec:marginal-mean-embedding}. 
 
\begin{proposition} \label{prop:transfer-rkhs}
Let $ p_{\scriptscriptstyle \inspace} $ be the reference density on $ \inspace $, with $ p_{\scriptscriptstyle \inspace}(x) > 0 $ for all $ x $, and let $ \mathcal{A}_k \colon \mathbb{H} \rightarrow \mathbb{H} $ be the kernel transfer operator with respect to this density, i.e.,
\begin{equation*}
    \mathcal{A}_k \ts g(y) = \frac{1}{p_{\scriptscriptstyle \inspace}(y)} \int p_\tau(y \mid x) \ts g(x) \ts p_{\scriptscriptstyle \inspace}(x) \ts \dd x.
\end{equation*}
Assume that $\mathcal{A}_k \ts g \in \mathbb{H}$ for all $ g  \in \mathbb{H}$. Then $ \cov[XX] \mathcal{A}_k \ts g = \cov[YX] g $.
\end{proposition}

\begin{proof}
The proof is similar to the proof of Proposition~\ref{thm:Fukumizu04}, which can be found, e.g., in~\citet{Fukumizu04}. Using~\eqref{eq:Cross-covariance}, it holds that
\begin{align*}
    \innerprod{f}{\cov[XX] \mathcal{A}_k \ts g}_{\mathbb{H}}
        &= \mathbb{E}_{\scriptscriptstyle X}[f(X) \ts \mathcal{A}_k \ts g(X)] \\
        &= \int f(y) \frac{1}{p_{\scriptscriptstyle \inspace}(y)} \int p_\tau(y \mid x) \ts p_{\scriptscriptstyle \inspace}(x) \ts g(x) \ts \dd x \ts p_{\scriptscriptstyle \inspace}(y) \ts \dd y \\
        &= \iint f(y) \ts g(x) \ts p_\tau(y \mid x) \ts p_{\scriptscriptstyle \inspace}(x) \ts \dd x \ts \dd y \\
        &= \iint f(y) \ts g(x) \ts p_\tau(x, y) \ts \dd x \ts \dd y \\
        &= \mathbb{E}_{\scriptscriptstyle XY}[g(X) \ts f(Y)] \\
        &= \innerprod{f}{\cov[YX] g}_{\mathbb{H}}. \qedhere
\end{align*}
\end{proof}

Based on Proposition~\ref{prop:transfer-rkhs}, we express the kernel transfer operator as $ \mathcal{A}_{k} = (\cov[XX] + \varepsilon\mathcal{I})^{-1} \ts \cov[YX] $. Let $ u_{\scriptscriptstyle \inspace} $ denote the uniform density on $ \inspace $ and $ \pi $ the invariant density. If $ p_{\scriptscriptstyle \inspace} = u_{\scriptscriptstyle \inspace} $, then $ \mathcal{A}_{k} = \pf[k] $ and if $ p_{\scriptscriptstyle \inspace} = \pi $, then $ \mathcal{A}_{k} = \mathcal{T}_{k} $, where $ \mathcal{T}_{k} $ denotes the kernel Perron--Frobenius operator with respect to the invariant density. It is important to note here that $ \inspace $ and $ \outspace $ as well as $ \mathbb{H} $ and $ \mathbb{G} $ have to be the same spaces, otherwise the operator would be undefined since $ \cov[YX] $ is a mapping from $ \mathbb{H} $ to $ \mathbb{G} $ and $ (\cov[XX] + \varepsilon \idop)^{-1} $ a mapping from $ \mathbb{H} $ to $ \mathbb{H} $.

\begin{corollary} \label{col:PF derivation}
For specific choices of $ p_{\scriptscriptstyle \inspace} $, we obtain:
\begin{enumerate}[label=(\roman*), itemsep=0ex, topsep=1ex]
\item If $ p_{\scriptscriptstyle \inspace} = u_{\scriptscriptstyle \inspace} $, then $ \pf[k] = (\cov[XX] + \varepsilon \idop)^{-1} \ts \cov[YX] $.
\item If $ p_{\scriptscriptstyle \inspace} = \pi $, then $ \mathcal{T}_{k} = (\cov[XX] + \varepsilon \idop)^{-1} \ts \cov[YX] $.
\end{enumerate}
\end{corollary}

This is consistent with the derivation of EDMD for the Perron--Frobenius operator using---from a kernel point of view---explicitly given finite-dimensional feature spaces. In this case, setting $\varepsilon=0$, the empirical estimates of the operators converge to a Galerkin approximation, see \citet{KKS16} for details. If the matrix $ \cov[XX] $ is not invertible, the pseudoinverse is used instead.

\begin{proposition} \label{pro:Estimation Perron-Frobenius}
The empirical estimate $ \epf[k] $ of the kernel Perron--Frobenius operator $ \pf[k] $ can be written as $ \epf[k] = \Psi A \Phi^\top$, where $ A = \gram[XY]^{-1} \ts (\gram[XX]+n \varepsilon \id)^{-1} \ts \gram[XY] $. Here, we assume that $ \gram[XY] $ is invertible.
\end{proposition}
\begin{proof}
The idea is to simply solve the equation
\begin{equation*}
    \epf[k] = (\ecov[XX] + \varepsilon \idop)^{-1} \ts \ecov[YX]
      = \left(\frac{1}{n} \Phi\Phi^\top + \varepsilon \idop\right)^{-1} \frac{1}{n} \Psi \Phi^\top 
      = \Psi A \Phi^\top
\end{equation*}
for $A$. Dropping the $\Phi^\top$ on the right, we multiply the equations from the left by $ \Phi^\top $, which leads to
\begin{equation*}
    \Phi^\top \left(\Phi\Phi^\top + \varepsilon n \idop\right)^{-1} \Psi
     = \left( \Phi^\top \Phi + \varepsilon n \id\right)^{-1} \Phi^\top \Psi
     = \Phi^\top \Psi A
\end{equation*}
and thus $ (\gram[XX] + n \varepsilon \id)^{-1} \gram[XY] = \gram[XY] A $. 
\end{proof}

Here, we  used the identity $ \Phi^\top \left(\Phi\Phi^\top + \varepsilon n \idop\right)^{-1} = \left( \Phi^\top \Phi + \varepsilon n \id\right)^{-1} \Phi^\top $, see also \cite{MFSS16}. With the aid of the reproducing property of $\mathbb{H}$ and assuming that $ p \in \mathbb{H} $, we can write
\begin{align*}
    \pf[k] \ts p(x)
        &= \innerprod{(\cov[XX]+\varepsilon \idop)^{-1} \ts \cov[YX] \ts p}{k(x, \cdot)}_\mathbb{H} \\
        &= \innerprod{p}{\cov[XY] \ts (\cov[XX] + \varepsilon \idop)^{-1} \ts k(x, \cdot)}_\mathbb{H} \\
        &= \innerprod{p}{\ko[\ebd] \ts k(x, \cdot)}_\mathbb{H},
\end{align*}
where $ \ko[\ebd] = \cov[XY] \ts (\cov[XX] + \varepsilon \idop)^{-1} $. Thus, the action of the Perron--Frobenius operator can be interpreted as an inner product in a Hilbert space.
We will call $ \ko[\ebd] $ the \emph{embedded Koopman operator} and discuss it in detail in Section~\ref{ssec:Embedded Koopman}.

\subsubsection{Embedded Perron--Frobenius Operator}
\label{ssec:Embedded Perron--Frobenius Operator}

Up to now, we assumed that the densities $ p_t $ and $ p_{t+\tau} $ are elements of the RKHS $ \mathbb{H} $. Now we first embed the densities into the RKHS $ \mathbb{H} $ using the mean embedding and consider the corresponding kernel mean embeddings $ \mu_t $ and $ \mu_{t+\tau} $. Let $ \mu_t = \ebd[k] \ts p_t $  be a Hilbert space embedding of the density $ p_t $, then the Perron--Frobenius operator for embedded densities can be expressed in terms of the conditional mean embedding $ \cme $ as
\begin{equation*}
    \mu_{t + \tau}
        = \cme \ts \mu_t
        = \cov[YX] \ts (\cov[XX] + \varepsilon\mathcal{I})^{-1} \ts \mu_t,
\end{equation*}
where $ \mu_{t+\tau} $ is the Hilbert space embedding of the density $ p_{t + \tau} $. The above equality is guaranteed under the assumption that $ \cov[XX] $ is injective, $\mu_t\in \mathrm{Range}(\cov[XX])$, and $\mathbb{E}_{Y|X}[g(Y) \mid X = \cdot \,] \in\mathbb{H}$ for all $g\in\mathbb{H}$ (see also \citealt{Fukumizu13:KBR}). Thus, we define $ \pf[\ebd] = \cme $ to be the embedded Perron--Frobenius operator, whose empirical estimate is given by
\begin{equation*}
    \epf[\ebd]
        = \ecov[YX] \ts (\ecov[XX] + \varepsilon\mathcal{I})^{-1} 
        = \left(\frac{1}{n}\Psi \Phi^\top\right) \left(\frac{1}{n} \Phi \Phi^\top + \varepsilon\mathcal{I}\right)^{-1}
        = \Psi \ts (\gram[XX] + n\varepsilon \id)^{-1} \Phi^\top.
\end{equation*}

\begin{proposition} \label{pro:Commutativity Perron-Frobenius}
Assume that $ \pf[\ebd] $ exists without necessitating regularization.\!\footnote{This holds, for instance, for finite domains equipped with characteristic kernels, but not necessarily for continuous domains \citep{Song2013}.}
Let $ p_t \in L^1 $ be a probability density and $ \mu_t = \ebd[k] \ts p_t $ the corresponding embedded density. Then the diagram
\begin{equation*}
    \begin{tikzcd}
        L^1 \ni \hspace*{-4em}
            & p_t \rar{\ebd[k]} \arrow[swap]{d}{\pf}
            & \mu_t \arrow{d}{\pf[\ebd]}
            & \hspace*{-4.5em} \in \mathbb{H} \\
        L^1 \ni \hspace*{-3em}
            & p_{t+\tau} \rar{\ebd[k]}
            & \mu_{t+\tau}
            & \hspace*{-3.5em} \in \mathbb{H}
    \end{tikzcd}
\end{equation*} 
is commutative.
\end{proposition}

\begin{proof}
Applying $ \pf $ to $ p_t $ and then embedding the resulting density leads to
\begin{equation*}
    \ebd[k](\pf p_t) = \int k(y, \cdot) \int p_\tau(y \mid x) \ts p_t(x) \ts \dd x \ts \dd y,
\end{equation*}
embedding $ p_t $ and then applying the embedded Perron--Frobenius operator to
\begin{align*}
    \pf[\ebd](\ebd[k] \ts  p_t)
        &= \pf[\ebd] \int k(x, \cdot) \ts p_t(x) \ts \dd x \\
        &= \int \pf[\ebd] \ts k(x, \cdot) \ts p_t(x) \ts \dd x \\
        &= \int \mathbb{E}_{\scriptscriptstyle Y \mid x}[\phi(Y) \mid X = x] \ts p_t(x) \ts \dd x \\
        &= \iint p_\tau(y \mid x) \ts \phi(y) \ts \dd y \ts p_t(x) \ts \dd x \\
        &= \int k(y, \cdot) \int p_\tau(y \mid x) \ts p_t(x) \ts \dd x \ts \dd y. \qedhere
\end{align*}
\end{proof}

If we assume that the feature space is finite-dimensional and $\ecov[XX]^{-1}$ exists without necessitating regularization, then the commutativity for the empirical estimates can be seen as follows: Let $ p_t $ be a probability density, then the empirical estimate of the kernel mean embedding of $ p_t $ is $ \widehat{\mu}_t = \frac{1}{n} \Phi \mathds{1} $. Applying $ \epf[\ebd] $ yields
\begin{equation*}
    \epf[\ebd] \ts \widehat{\mu}_{t}
        = \tfrac{1}{n} \ecov[YX] \ts \ecov[XX]^{-1} \Phi \mathds{1}
        = \tfrac{1}{n} \Psi \Phi^\top (\Phi \Phi^\top)^{-1} \Phi \mathds{1}
        = \tfrac{1}{n} \Psi \mathds{1}
        = \widehat{\mu}_{t+\tau}.
\end{equation*}
That is, we obtain the empirical estimate of the mean embedding of the density $ p_{t+\tau} $.

\begin{example}
Let us consider the Ornstein--Uhlenbeck process from Example~\ref{ex:Ornstein-Uhlenbeck}. We choose $ \tau = \frac{1}{2} $, $ \alpha = 4 $, $ D = \frac{1}{4} $, and the Gaussian kernel with $ \sigma^2 = \frac{1}{2} $. Figure~\ref{fig:OU embedding}\,a shows the piecewise constant initial probability density $ p_0 $ pushed forward by the Perron--Frobenius operator, Figure~\ref{fig:OU embedding}\,b the embedded initial density $ \mu_0 $ pushed forward by the embedded Perron--Frobenius operator. While the smoother embedded function $ \mu_0 $ can be easily approximated using only a few (e.g., uniformly distributed) data points, representing the function $ p_0 $ accurately as a weighted sum of Gaussians is challenging. This is illustrated in Figure~\ref{fig:OU embedding}\,c and \ref{fig:OU embedding}\,d.

 \exampleSymbol

\begin{figure}[tb]
    \centering
    \begin{minipage}{0.49\textwidth}
        \centering
        \subfiguretitle{a)}
        \includegraphics[width=0.9\textwidth]{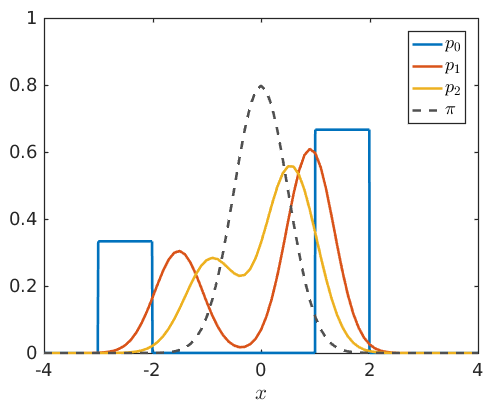}
    \end{minipage}
    \begin{minipage}{0.49\textwidth}
        \centering
        \subfiguretitle{b)}
        \includegraphics[width=0.9\textwidth]{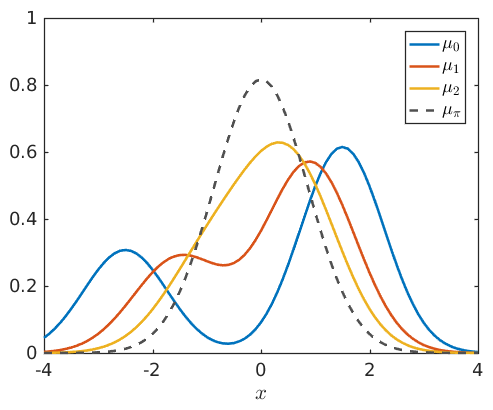}
    \end{minipage}
    \begin{minipage}{0.49\textwidth}
        \centering
        \subfiguretitle{c)}
        \includegraphics[width=0.9\textwidth]{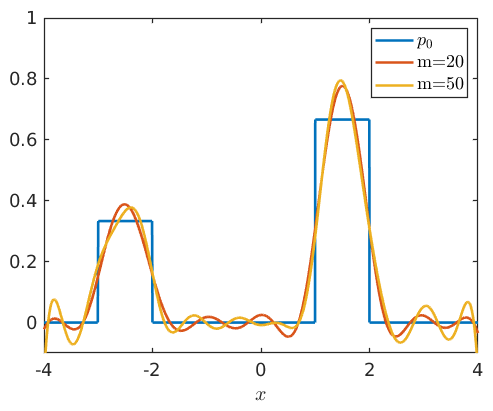}
    \end{minipage}
    \begin{minipage}{0.49\textwidth}
        \centering
        \subfiguretitle{d)}
        \includegraphics[width=0.9\textwidth]{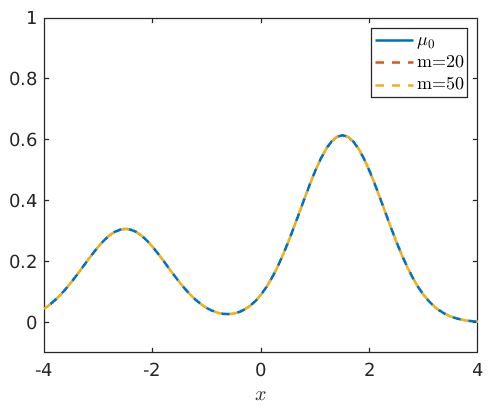}
    \end{minipage}
    \caption{a) Propagation of the initial density $ p_0 $ by the Perron--Frobenius operator, where $ p_1 = \pf  p_0 $ and $ p_2 = \pf p_1 $. b)~Propagation of the embedded density $ \mu_0 $ by the embedded Perron--Frobenius operator, where $ \mu_1 = \pf[\ebd] \mu_0 $ and $ \mu_2 = \pf[\ebd] \mu_1 $. The dashed black lines show the invariant and embedded invariant density, respectively. c--d) Best approximation of $ p_0 $ and $ \mu_0 $ using $ m $ uniformly distributed test points and the Gaussian kernel with bandwidth $ \sigma^2 = \frac{1}{2} $.}
    \label{fig:OU embedding}
\end{figure}

\end{example}

The Perron--Frobenius operator $ \pf $ maps densities $ p_t \in L^1 $ to $ p_{t+\tau} \in L^1 $, while the embedded Perron--Frobenius operator $ \pf[\ebd] $, given by the conditional mean embedding $ \cme $, maps the embeddings $ \mu_t \in \mathbb{H} $ to $ \mu_{t+\tau} \in \mathbb{H} $. Thus, the conditional mean embedding plays a similar role as the classical Perron--Frobenius operator in that it pushes forward---in this case: embedded---densities. Note that if we embed an eigenfunction of $ \pf $, we automatically obtain an eigenfunction of $ \pf[\ebd] $. This is due to the linearity of the integral.

\begin{remark}
An open question is for what classes of dynamical systems the embedded eigenfunctions are easier to approximate and whether we can directly use the embedded eigenfunctions to detect, e.g., metastable states and lower-dimensional representations or whether we have to compute the preimages of the embedded functions. The problem of reconstructing probability densities or observables from their embedded counterparts is addressed in \cite{SMKM19}.
\end{remark}

\subsection{Koopman Operators}
\label{ssec:Koopman}

Instead of the push-forward of densities or kernel mean embeddings we now consider the push-forward of observables.

\subsubsection{Kernel Koopman Operator}

Like the kernel Perron--Frobenius operator, we now introduce the corresponding \emph{kernel Koopman operator}, denoted by $ \ko[k] $. That is, we assume that the observables mapped forward by the Koopman operator are elements of $ \mathbb{H} $. From Proposition~\ref{thm:Fukumizu04}, it follows  that $ \ko[k] = (\cov[XX] + \varepsilon\mathcal{I})^{-1} \ts \cov[XY] $ and thus for all $ f \in \mathbb{H} $ that
\begin{equation} \label{eq:koopman-adjoint}
   \left(\ko[k] f\right)(x)
        = \innerprod{(\cov[XX] + \varepsilon\mathcal{I})^{-1} \ts \cov[XY] f}{k(x, \cdot)}
        = \innerprod{f}{\pf[\ebd] \ts k(x, \cdot)}.
\end{equation}
The empirical estimate of the Koopman operator is then given by
\begin{equation*}
    \eko[k]
        = (\ecov[XX] + \varepsilon\mathcal{I})^{-1} \ts \ecov[XY]
        = \left(\frac{1}{n} \Phi \Phi^\top + \varepsilon\mathcal{I}\right)^{-1} \left(\frac{1}{n} \Phi \Psi^\top\right)
        = \Phi \ts (\gram[XX] + n \varepsilon \id)^{-1} \Psi^\top.
\end{equation*}

\color{black}
\begin{example} \label{ex:EDMD simple system}
Let us approximate the kernel Koopman operator associated with the system defined in Example~\ref{ex:Simple system} using the explicit feature space representation of the kernel from Example~\ref{ex:Kernel}. Generating $ 10000 $ test points $ x_i $ sampled from a uniform distribution on $ \inspace = [-2,\, 2] \times [-2,\, 2] $ and the corresponding $ y_i = F(x_i) $ values and setting $ \varepsilon = 0 $, we can compute the empirical estimator $ \eko[k] = \big(\Phi \Phi^\top\big)^{-1} \big(\Phi \Psi^\top \big) = \big(\Psi \Phi^+\big)^\top \in \R^{6 \times 6} $. Here, $ ^+ $ denotes the pseudoinverse. The dominant eigenvalues and right eigenvectors as well as the corresponding eigenfunctions are given by
\begin{equation*}
    \begin{array}{l@{\quad}l@{\,}llllll@{\,}l@{\quad}l}
        \lambda_1 = 1.0,
            & v_1 = [& 1 & 0 & 0 & 0 & 0 & 0 & ]^\top,
            & \varphi_1(x) = \innerprod{v_1}{\phi(x)} = 1, \\
        \lambda_2 = 0.8,
            & v_2 = [& 0 & 0.7071 & 0 & 0 & 0 & 0 & ]^\top,
            & \varphi_2(x) = \innerprod{v_2}{\phi(x)} \approx x_1, \\
        \lambda_3 = 0.7,
            & v_3 = [& 0 & 0 & 0.7071 & 1 & 0 & 0 & ]^\top,
            & \varphi_3(x) = \innerprod{v_3}{\phi(x)} \approx x_2 + x_1^2.
    \end{array}
\end{equation*} 
This is in good agreement with the analytically computed results. The subsequent eigenvalues and eigenfunctions are simply products of the eigenvalues and eigenfunctions listed above, see Example~\ref{ex:Simple system}. Note, however, that other than $ \varphi_4 $ further products of eigenfunctions cannot be represented as functions in $ \mathbb{H} $ anymore since the feature space does not contain polynomials of order greater than two. \exampleSymbol
\end{example}

The approach to obtain an approximation of transfer operators from data as described in the above example is also referred to as EDMD~\citep{WKR15, KKS16}. Details regarding the relationships with other methods can be found in Section~\ref{ssec:Relationships with other methods} and further examples in Section~\ref{sec:experiments}.

\subsubsection{Embedded Koopman Operator}
\label{ssec:Embedded Koopman}

If we want to embed the Koopman operator in the same way as the Perron--Frobenius operator, we need to introduce embedded observables first, which can be interpreted as the \emph{Koopman analogue} of the mean embedding of distributions. Let $ f \colon \inspace \to \R $ be an observable of the system. We define $\nu := \ebd[k] f $ to be the \emph{embedded observable}. Given a set of training data, the empirical estimate $ \widehat{\nu} $ of the embedded observable is given by
\begin{equation*}
    \widehat{\nu}
        = \frac{1}{n} \sum_{i=1}^n \phi(x_i) \ts f(x_i)
        = \frac{1}{n} \sum_{i=1}^n k(x_i, \cdot) \ts f(x_i)
        = \frac{1}{n} \Phi \widehat{f},
\end{equation*}
where $ \widehat{f} = [f(x_1), \, \dots, \, f(x_n)]^\top $ contains the values of the observable evaluated at the training data points. Note that the data points do not have to be drawn from a particular probability distribution. Analogously to the embedded Perron--Frobenius operator, we define the \emph{embedded Koopman operator} by $ \ko[\ebd] = \cov[XY] \ts (\cov[XX] + \varepsilon\mathcal{I})^{-1} $.

\begin{proposition} \label{pro:Commutativity Koopman}
Analogously to Proposition~\ref{pro:Commutativity Perron-Frobenius}, we assume that $ \ko[\ebd] $ exists without necessitating regularization. Let $ f_t $ be an observable and $ \nu_t = \ebd[k] f_t $ the corresponding embedded observable. Then the diagram
\begin{equation*}
    \begin{tikzcd}
        L^\infty \ni \hspace*{-4em}
            & f_t \rar{\ebd[k]} \arrow[swap]{d}{\ko}
            & \nu_t \arrow{d}{\ko[\ebd]}
            & \hspace*{-4.5em} \in \mathbb{H} \\
        L^\infty \ni \hspace*{-3em}
            & f_{t+\tau} \rar{\ebd[k]}
            & \nu_{t+\tau}
            & \hspace*{-3.5em} \in \mathbb{H}
    \end{tikzcd}
\end{equation*}
is commutative.
\end{proposition}

The proof is similar to the proof of Proposition~\ref{pro:Commutativity Perron-Frobenius}. As for $ \pf $, embedding an eigenfunction of $ \ko $ results in an eigenfunction of $ \ko[\ebd] $. We assume now that the feature space is  finite-dimensional and $ \cov[XX]^{-1} $ exists, i.e., we set $ \varepsilon = 0 $. For the kernel Koopman operator $ \ko[k] $, the commutativity can then be seen as follows: Given an observable $ f_t \in \mathbb{H} $, then $ \ko[k] f_t = \cov[XX]^{-1} \ts \cov[XY] f_t $, while the embedding of $ f_t $ results in $ \nu_t = \cov[XX] f_t $. Applying $ \ko[\ebd] $, this results in
\begin{equation*}
    \ko[\ebd] (\ebd[k] f_t)
        = \cov[XY] \ts \cov[XX]^{-1} \ts \cov[XX] f_t
        = \cov[XX] \ts \cov[XX]^{-1} \ts \cov[XY] f_t
        = \ebd[k] (\ko[k] f_t).
\end{equation*}

\begin{proposition} \label{pro:Estimation embedded Koopman}
The empirical estimate $ \eko[\ebd] $ of the embedded Koopman operator $ \ko[\ebd] $ can be written as $ \eko[\ebd] = \Phi A \Psi^\top$, where $ A = \gram[YX] (\gram[XX] + n \varepsilon \id)^{-1} \gram[YX]^{-1} $. We assume again that $ \gram[YX] $ is invertible.
\end{proposition}

The proof is analogous to the proof of Proposition~\ref{pro:Estimation Perron-Frobenius}.

\begin{example}
Let us consider again the system from Example~\ref{ex:Simple system} whose eigenfunctions $ \varphi_1 $, $ \varphi_2 $, and $ \varphi_3 $ we estimated numerically in Example~\ref{ex:EDMD simple system}. Computing the corresponding embedded eigenfunctions analytically, we obtain the (properly rescaled) functions $ \nu_1 $, $ \nu_2 $, and $ \nu_3 $ and associated vector representations $ w_1 $, $ w_2 $, and $ w_3 $, given by
\begin{equation*}
    \begin{array}{l@{\quad}l@{\,}cccccc@{\,}l@{\quad}l}
        \lambda_1 = 1.0,
            & w_1 = [& 3 & 0 & 0 & 4 & 0 & 4 & ]^\top,
            & \nu_1(x) = 3 + 4 \ts x_1^2 + 4 \ts x_2^2, \\
        \lambda_2 = 0.8,
            & w_2 = [& 0 & 1 & 0 & 0 & 0 & 0 & ]^\top,
            & \nu_2(x) = \sqrt{2}\ts x_1, \\
        \lambda_3 = 0.7,
            & w_3 = [& 1 & 0 & \sqrt{2} & \frac{12}{5} & 0 & \frac{4}{3} & ]^\top,
            & \nu_3(x) = 1 + 2 \ts z_2 + \frac{12}{5} z_1^2 + \frac{4}{3} z_2^2. 
    \end{array}
\end{equation*}
For $ \varepsilon = 0 $, the vectors $ w_1 $, $ w_2 $, and $ w_3 $ are indeed eigenvectors of the matrix $ \eko[\ebd] = \ecov[XY] \ts \ecov[XX]^{-1} $ corresponding to the eigenvalues $ \lambda_1 $, $ \lambda_2 $, and $ \lambda_3 $. \exampleSymbol
\end{example}

%%%
\subsection{Consistency and Convergence Rate}
\label{sec:consistency}

In this section we establish the consistency and convergence rate of the proposed estimators. Note that we only emphasize on bounding the estimation error and leave a complete convergence analysis of the kernel transfer operators to future work.

%%%%  
\begin{theorem}\label{thm:consistency-results}
  Let $\mathcal{S}$ be either of $\pf[\ebd], \pf[k], \ko[\ebd]$ and $\ko[k]$. Assume that the RKHS $\mathbb{H}$  it acts upon is separable and endowed with a measurable kernel $k$. For a (linear) bounded operator $B:\mathbb{H}\to\mathbb{H}$, we denote an operator norm on $\mathbb{H}$ by $\|B\| := \sup_{\|f\|=1}\|Bf\|$. Suppose that $\mathbb{E}_{X_t}[k(X_t,X_t)] < \infty$ and $\mathbb{E}_{X_{t+\tau}}[k(X_{t+\tau},X_{t+\tau})] < \infty$. Then, for all positive regularization parameters $\varepsilon$, the empirical estimate $\widehat{\mathcal{S}} $ converges to $\mathcal{S}$  at rate
  \begin{equation}
    \left\| \widehat{\mathcal{S}} - \mathcal{S}\right\| = O_p\left(n^{-1/2}\varepsilon^{-1}\right), \qquad (n\to\infty).
  \end{equation}
\end{theorem}
 
\begin{proof}
We will first show the result for $\mathcal{S} = \pf[\ebd]$. By the definitions of $\pf[\ebd]$ and $\epf[\ebd]$ and the triangle inequality, we have
\begin{eqnarray}
  \left\| \epf[\ebd] - \pf[\ebd]\right\|^2 &=& \left\|\ecov[YX](\ecov[XX] + \varepsilon\mathcal{I})^{-1} - \cov[YX](\cov[XX] + \varepsilon\mathcal{I})^{-1}\right\|^2 \nonumber \\
  &=& \Big\|\ecov[YX](\ecov[XX] + \varepsilon\mathcal{I})^{-1} - \cov[YX](\ecov[XX] + \varepsilon\mathcal{I})^{-1} \nonumber \\
  && + \cov[YX](\ecov[XX] + \varepsilon\mathcal{I})^{-1} - \cov[YX](\cov[XX] + \varepsilon\mathcal{I})^{-1}\Big\|^2 \nonumber \\
  &\leq& \left\|\ecov[YX](\ecov[XX] + \varepsilon\mathcal{I})^{-1} - \cov[YX](\ecov[XX] + \varepsilon\mathcal{I})^{-1} \right\|^2 \nonumber \\
  && + \left\|\cov[YX](\ecov[XX] + \varepsilon\mathcal{I})^{-1} - \cov[YX](\cov[XX] + \varepsilon\mathcal{I})^{-1}\right\|^2 \label{eq:bound1}.
\end{eqnarray} 
Due to the separability of $\mathbb{H}$, there exists a Hilbert basis $(e_j)_{j\geq 1}$ of $\mathbb{H}$. Let $\langle A,B\rangle_{\text{HS}} = \sum_{j\geq 1}\langle Ae_j,Be_j\rangle$ be the scalar product associated with the space of Hilbert--Schmidt operators on $\mathbb{H}$ and $\| \cdot \|_{\text{HS}}$ the corresponding norm. Since both $\cov[XX]$ and $\cov[YX]$ are Hilbert--Schmidt operators on $\mathbb{H}$ \citep[Section 3.2]{MFSS16}, we can bound the first term on the right-hand side of \eqref{eq:bound1} as follows:
\begin{eqnarray*}
  \left\|\ecov[YX](\ecov[XX] + \varepsilon\mathcal{I})^{-1} - \cov[YX](\ecov[XX] + \varepsilon\mathcal{I})^{-1} \right\|^2 &\leq& \left\|\ecov[YX] - \cov[YX] \right\|^2\left\|(\ecov[XX] + \varepsilon\mathcal{I})^{-1} \right\|^2 \\
  &\leq& \left\|\ecov[YX] - \cov[YX] \right\|_{\text{HS}}^2\left\|(\ecov[XX] + \varepsilon\mathcal{I})^{-1} \right\|^2.
\end{eqnarray*}
For the second inequality, we used the fact that $\|B\| \leq \|B\|_{\text{HS}}$ for a Hilbert--Schmidt operator $B$. Then, it follows from \citet[Lemma 4]{Fukumizu07:SCK} that $\|\ecov[YX] - \cov[YX] \|_{\text{HS}}^2 = O_p(1/n)$. Moreover, we have $\|(\ecov[XX] + \varepsilon\mathcal{I})^{-1} \|^2 \leq 1/\varepsilon^2$. Hence, it follows that
\begin{equation}\label{eq:first-term}
  \left\|\ecov[YX] - \cov[YX] \right\|_{\text{HS}}^2\left\|(\ecov[XX] + \varepsilon\mathcal{I})^{-1} \right\|^2 = O_p\left(\frac{1}{n\varepsilon^2}\right), \quad (n\to\infty).
\end{equation}
Next, we bound the second term on the right-hand side of \eqref{eq:bound1}. Likewise, it can be bounded from above by $\|\cov[YX]\|^2\|(\ecov[XX] + \varepsilon\mathcal{I})^{-1} - (\cov[XX] + \varepsilon\mathcal{I})^{-1}\|^2$. By using the fact that $B^{-1} - D^{-1} = B^{-1}(D-B)D^{-1}$ holds for any invertible operators $B$ and $D$, we have
\begin{eqnarray} 
  \left\|(\ecov[XX] + \varepsilon\mathcal{I})^{-1} - (\cov[XX] + \varepsilon\mathcal{I})^{-1}\right\|^2 &=& \left\|(\ecov[XX] + \varepsilon\mathcal{I})^{-1}(\cov[XX]-\ecov[XX])(\cov[XX] + \varepsilon\mathcal{I})^{-1}\right\|^2 \nonumber \\
  &\leq& \frac{1}{\varepsilon^4} \left\|\cov[XX]-\ecov[XX]\right\|^2 = O_p\left(\frac{1}{n\varepsilon^4} \right), \quad (n\to\infty), \label{eq:second-term} 
\end{eqnarray}
where the last equality again follows from the fact that $\|B\| \leq \|B\|_{\text{HS}}$ and \citet[Lemma 4]{Fukumizu07:SCK}. Combining \eqref{eq:first-term} and \eqref{eq:second-term} completes the proof for $\pf[\ebd]$. This implies convergence of $\ko[k]$ from the fact that it is the adjoint of $\pf[\ebd]$ (see \eqref{eq:koopman-adjoint}). Furthermore, $\pf[k]$ and $\ko[\ebd]$ are adjoint to each other and their empirical estimates can be expressed as $(\ecov[XX] + \varepsilon\mathcal{I})^{-1}\ecov[YX]$ and $\ecov[XY](\ecov[XX] + \varepsilon\mathcal{I})^{-1}$, respectively, allowing for an analogous derivation.
\end{proof}

\begin{remark} Two remarks are in order:
\begin{enumerate}[label=(\roman*), wide, itemindent=0.5em]
\item For an approximation error to vanish, $\varepsilon$ must also decay to zero. It follows from Theorem \ref{thm:consistency-results} that the decay rate must be slower than $1/\sqrt{n}$. Setting $\varepsilon = n^{-\alpha}$ where $\alpha < 1/2$ yields an overall rate of $n^{\alpha-1/2}$. We leave detailed analysis of the approximation error to future work.

\item The result differs from existing convergence results of the conditional mean embedding \citep{SHSF09,Grunewalder12:LGBPP,Song2013} in that we consider convergence in an operator norm rather than an RKHS norm. 

\end{enumerate}
\end{remark}

Theorem \ref{thm:consistency-results} shows that the kernel transfer operators can be estimated \emph{consistently} and at a reasonably fast rate from the empirical data without requiring parametric assumptions about the underlying dynamical system.

\subsection{Relationships with Other Methods}
\label{ssec:Relationships with other methods}

There are several existing methods such as \emph{time-lagged independent component analysis} (TICA) \citep{MS94, PPGDN13}, \emph{dynamic mode decomposition} (DMD) \citep{Schmid10, TRLBK14}, and their respective generalizations---the aforementioned VAC and EDMD---to approximate transfer operators and their eigenvalues, eigenfunctions, and eigenmodes. Although developed independently from each other, these methods are strongly related as shown in \citet{KNKWKSN18}. Our methods subsume existing ones and thereby provide a unified framework for transfer operator approximation using RKHS theory. Recently, a data-driven approach for the spectral analysis of measure-preserving ergodic dynamical systems using RKHS theory was also proposed in \citet{GDS18}.

\subsubsection{TICA and DMD}

TICA can be used to separate superimposed signals \citep{MS94}, solving the so-called \emph{blind source separation} problem, and also for dimensionality reduction \citep{PPGDN13}, by projecting a high-dimensional signal onto the main TICA coordinates (see Section~\ref{sec:experiments} for an example). The method aims to find the time-lagged independent components that are uncorrelated and maximize the autocovariances at lag time $ \tau $. Given again training data $ \mathbb{D}_{\scriptscriptstyle XY} = \{(x_1, y_1), \dots, (x_n, y_n)\} $, where $x_i = X_{t_i}$ and $y_i = X_{t_i+\tau}$, we define the associated data matrices $ \mathbf{X}, \mathbf{Y} \in \R^{d\times n}$ by
\begin{equation*} \label{eq:data matrices}
    \mathbf{X} =
    \begin{bmatrix}
        x_1 & \cdots & x_n
    \end{bmatrix}
    \quad \text{and} \quad
    \mathbf{Y} =
    \begin{bmatrix}
        y_1 & \cdots & y_n
    \end{bmatrix}.
\end{equation*}
By setting $ k(x,x^\prime) = x^\top x^\prime $ and $ l(y,y^\prime) = y^\top y^\prime $, the
eigenvalue problem for the Koopman operator reduces to the standard eigenvalue problem $ \ecov[XX]^{-1} \ts \ecov[XY] \ts \xi = \lambda \ts \xi $, where $ \ecov[XX] $ and $ \ecov[XY] $ denote the covariance and cross-covariance matrices, respectively, defined by $ \ecov[XX] = \frac{1}{n} \mathbf{X} \mathbf{X}^\top $ and $ \ecov[XY] = \frac{1}{n} \mathbf{X} \mathbf{Y}^\top $. The resulting eigenvectors are defined to be the TICA coordinates.

DMD is frequently used for the analysis of high-dimensional fluid flow problems \citep{Schmid10}. The DMD modes correspond to coherent structures in these flows. The derivation is based on the least-squares minimization problem $ \norm{\mathbf{Y} - M \mathbf{X}}_F $, whose solution is given by
\begin{equation*}
    M = \mathbf{Y} \mathbf{X}^+
      = \big(\mathbf{Y} \mathbf{X}^\top\big)\big(\mathbf{X} \mathbf{X}^\top\big)^{-1}
      = \ecov[YX] \ts \ecov[XX]^{-1}.
\end{equation*}
Eigenvectors of this matrix are then called DMD modes. Equivalently, the DMD modes can be interpreted as the left eigenvectors of the TICA matrix $ \ecov[XX]^{-1} \ts \ecov[XY] $. More details on the relationships between TICA and DMD can be found in \citet{KNKWKSN18}. As shown above, both TICA and DMD can be obtained as special cases of our algorithms.

\subsubsection{VAC and EDMD}

For a given set of basis functions $ \phi_1, \dots, \phi_r $, we define the vector-valued function $ \phi = [\phi_1, \dots, \phi_r]^\top \colon \R^d \to \R^r $. In the context of the kernel-based methods introduced above, the function $ \phi $ corresponds to an explicitly defined feature map. This results in the feature matrices $ \Phi, \Psi \in \R^{r \times n} $, given by
\begin{equation*}
    \Phi =
    \begin{bmatrix}
        \phi(x_1) & \cdots & \phi(x_n)
    \end{bmatrix}
    \quad \text{and} \quad
    \Psi =
    \begin{bmatrix}
        \phi(y_1) & \cdots & \phi(y_n)
    \end{bmatrix}.
\end{equation*}
VAC and EDMD, which are equivalent for reversible dynamical systems, can be understood as nonlinear extensions of TICA and DMD, respectively. Both methods utilize the transformed data matrices $ \Phi $ and $ \Psi $ for an explicitly given set of basis functions. VAC uses the matrix $ \ecov[XX]^{-1} \ts \ecov[XY] $ as an approximation of $ \mathcal{T} $ (which is equivalent to $ \ko $ for a reversible system) to compute eigenfunctions. Similarly, EDMD considers the matrix $ \ecov[YX] \ts \ecov[XX]^{-1} $, which can be interpreted as a least-square approximation of the Koopman operator using the transformed data matrices. (In the same way, we obtain $ \ecov[XY] \ts \ecov[XX]^{-1} $ for the Perron--Frobenius operator.) By defining the kernels $ k $ and $ l $ explicitly as $ k(x, x^\prime) = \phi(x)^\top \phi(x^\prime)$ and $ l(y, y^\prime) = \phi(y)^\top \phi(y^\prime) $ for some finite-dimensional feature spaces $ \mathbb{H} $ and $ \mathbb{G} $, we can also see the close relationship between the methods described in this paper and VAC and EDMD. Given a finite-dimensional feature space, $\ecov[XX] = \frac{1}{n}\Phi\Phi^\top$ and $\ecov[YX] = \frac{1}{n}\Psi\Phi^\top$ can be computed explicitly. In \citet{BGH15}, smooth orthogonal basis functions $ \phi_i $ are obtained by using diffusion maps. Then, the Galerkin projection coefficients are computed as $ \innerprod{\phi_\ell}{e^{\tau \mathcal{L}} \phi_{\ell'}} \approx \sum_{i=1}^n \phi_\ell(x_i) \ts \phi_{\ell'}(y_i) $, where $ \mathcal{L} $ is the generator of a drift-diffusion process. This corresponds to the matrix $ \ecov[XY] $ while $ \ecov[XX] $ in this case is the identity matrix. While the formulas are related, their representation is valid in the $ L_2 $-sense and does not necessitate any RKHS representation.

\subsubsection{Kernel TICA and Kernel EDMD}

The advantage of our method compared to VAC and EDMD is that the eigenvalue problem can be expressed entirely in terms of the Gram matrices $ \gram[XX] $, $ \gram[XY] $, $ \gram[YX] $, and $ \gram[YY] $. The transformed data matrices $ \Phi $ and $ \Psi $ need not be computed explicitly. This also allows us to work implicitly with infinite-dimensional feature spaces. Kernel-based variants, based on algebraic transformations of the conventional counterparts, of TICA and EDMD have also been proposed in \citet{SP15, WRK15} and DMD with reproducing kernels in \citet{Kawa16}. Although kernel TICA and kernel EDMD are generalizations of different methods---TICA is related to DMD and VAC to EDMD---, the resulting methods are strongly related again. In \citet{SP15}, conventional TICA is first implicitly extended to VAC and then to kernel TICA, whereas the derivation of kernel EDMD explicitly uses the EDMD feature space representation.

\subsubsection{Conditional Mean Embedding}

As mentioned earlier, the embedded Perron--Frobenius operator $\mathcal{P}_{\mathcal{E}} = \cov[YX]\cov[XX]^{-1}$ has indeed the same form as the conditional mean embedding $\mathcal{U}_{Y|X}$ of $\mathbb{P}(Y\,|\, X)$ which---in the context of this work---describes the dynamics of the system. The eigendecomposition presented in Section \ref{sec:eigendecomposition} for this operator may therefore be viewed in a similar vein as kernel PCA \citep{Scholkopf98:KPCA} for conditional distributions. Moreover, it follows from \eqref{eq:koopman-adjoint} that $\mathcal{P}_{\mathcal{E}}$ and $\mathcal{K}_k$ are adjoint to each other. Hence, their eigenvalues and eigenfunctions are related.

\section{Experiments}
\label{sec:experiments}

In what follows, we present several experiments that demonstrate the benefits of kernel transfer operators in exploratory studies of non-linear dynamical systems. We focus on applications in molecular dynamics, time-series analysis, and text analysis. The experiments were originally performed using Matlab, but most of the methods have been reimplemented in Python and are available at \url{https://github.com/sklus/d3s/}.

\subsection{Molecular Dynamics}

In this section, we apply the proposed techniques to extract meta-stable sets and to reduce the dimension of time-series data. More complex molecular dynamics examples can be found in \citet{KBSS18}.

\paragraph{Meta-stable sets.}

As a first example, let us illustrate how the eigendecomposition of $\ko[k]$---for an explicitly defined feature space, which corresponds to EDMD as described above---can be used for molecular dynamics applications. We consider a simple multi-well diffusion process given by a stochastic differential equation of the form
\begin{equation*}
    \dd X_t = -\nabla V(X_{t}) \ts \dd t + \sqrt{2 D} \, \dd W_t,
\end{equation*}
where $ V $ is the potential, $D = \beta^{-1}$ again the diffusion coefficient, and $W_t$ a standard Wiener process. The potential, taken from \citet{BKKBDS17} and visualized in Figure~\ref{fig:LemonSlice}\,a, is given by
\begin{equation*}
     V(x) = \cos\left(a \ts \arctan(x_2, x_1)\right) + 10 \left(\sqrt{x_1^2 + x_2^2} - 1\right)^2.
\end{equation*}
We set $ a = 5 $. A particle will typically spend a long time in one of the wells and then jump to one of the adjacent wells. The transitions between the wells are rare events. Thus, this system exhibits metastable behavior and the five metastable sets---which are encoded in the five dominant eigenfunctions of the transfer operators associated with the system---correspond to the five wells of the potential.

\begin{figure}[tbp]
    \centering
    \begin{minipage}{0.49\textwidth}
        \centering
        \subfiguretitle{a)}
        \includegraphics[height=0.27\textheight]{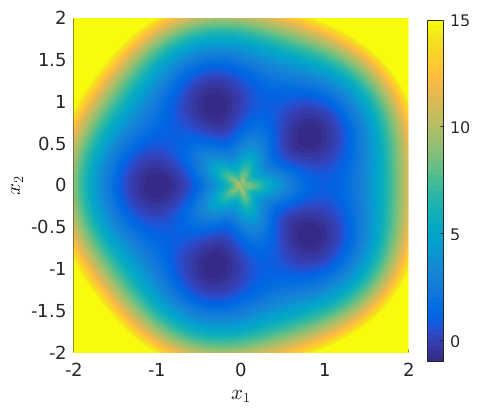}
    \end{minipage}
    \begin{minipage}{0.49\textwidth}
        \centering
        \subfiguretitle{b)}
        \includegraphics[height=0.27\textheight]{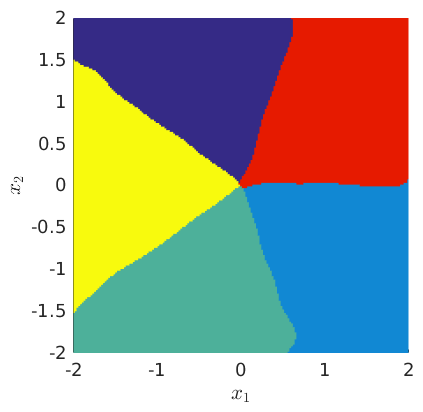}
    \end{minipage}
    \caption{a)~Potential $ V $ associated with the multi-well diffusion process. b)~Partitioning of the state space based on the dominant eigenfunctions of the Koopman operator.}
    \label{fig:LemonSlice}
\end{figure}

We use a $ 50 \times 50 $ box discretization of the domain $ \inspace = [-2, 2] \times [-2, 2] $ to define a basis containing $ 2500 $ radial basis functions $ k_i(x, c_i) = \exp(-\frac{1}{2 \sigma^2}\norm{x - c_i}^2) $ whose centers $ c_i $ are the centers of the boxes. This defines a kernel
\begin{equation*}
    k(x, x^\prime) = \sum_{i=1}^{2500} k_i(x, c_i) \ts k_i(x^\prime, c_i) = \sum_{i=1}^{2500} \exp\left(-\frac{1}{2 \sigma^2} \left(\norm{x - c_i}^2 + \norm{x^\prime - c_i}^2\right)\right).
\end{equation*}
Furthermore, we choose the lag time $ \tau = 0.2 $ and $ \sigma^2 = 0.9 $. We generate $ 250000 $ uniformly distributed test points $ x_i \in \inspace $ and solve the initial value problem with the Euler--Maruyama\footnote{See, e.g., \cite{kloeden2011numerical}.} method to obtain the corresponding $ y_i $ values. We then compute the eigenvalues and eigenfunctions of the Koopman operator $ \ko[k] $. There exist five dominant eigenvalues close to one and then there is a spectral gap between the fifth and sixth eigenvalue. We apply a $k$-means clustering to the dominant eigenfunctions to obtain the partitioning of the domain into the five metastable sets shown in Figure~\ref{fig:LemonSlice}\,b.

\paragraph{Dimensionality reduction and blind source separation.}

Another use case of the methods introduced above is dimensionality reduction. Before methods to compute eigenfunctions of transfer operators such as EDMD or VAC can be applied to high-dimensional systems, the data often needs to be projected onto a lower-dimensional subspace first. This can be accomplished by using TICA. Let us consider the simple data set $ x \in \R^{4 \times 10000} $ shown in Figure~\ref{fig:TICA}\,a. From this data set, we extract $ X = [x_1,\,\dots,\,x_{9999}] $ and $ Y = [x_2,\,\dots,\,x_{10000}] $, where $ x_i $ denotes the $i$th column vector of $ x $. Applying TICA, we see that there are two dominant eigenvalues close to $ 1 $, the other two are close to $ 0 $. This indicates that two of the four variables exhibit metastable behavior. Projecting the data onto the TICA coordinates results in the trajectories shown in Figure~\ref{fig:TICA}\,b. The first two new variables corresponding to the dominant eigenvalues contain the metastability, while the other two variables contain just noise. (In fact, this is how the data set was constructed.) Since we are only interested in the slow metastable dynamics, we can neglect the last two variables and thus reduce the state space. TICA corresponds to approximating the eigenfunctions of the kernel Koopman operator $\ko[k]$ using a linear kernel, see Section~\ref{ssec:Relationships with other methods} for details.

\begin{figure}[tbhp]
    \centering
    \begin{minipage}{0.46\textwidth}
        \centering
        \subfiguretitle{a)}
        \includegraphics[width=\textwidth]{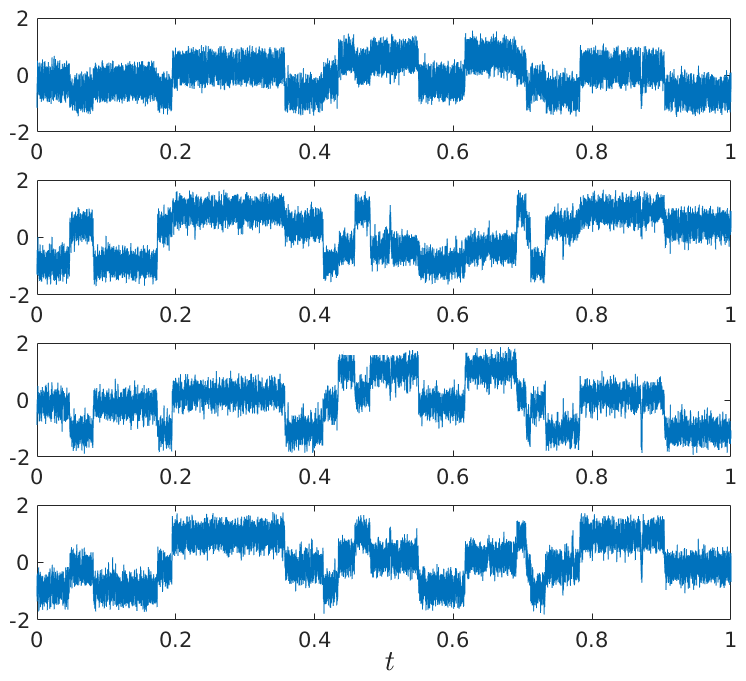}
    \end{minipage}
    \begin{minipage}{0.46\textwidth}
        \centering
        \subfiguretitle{b)}
        \includegraphics[width=\textwidth]{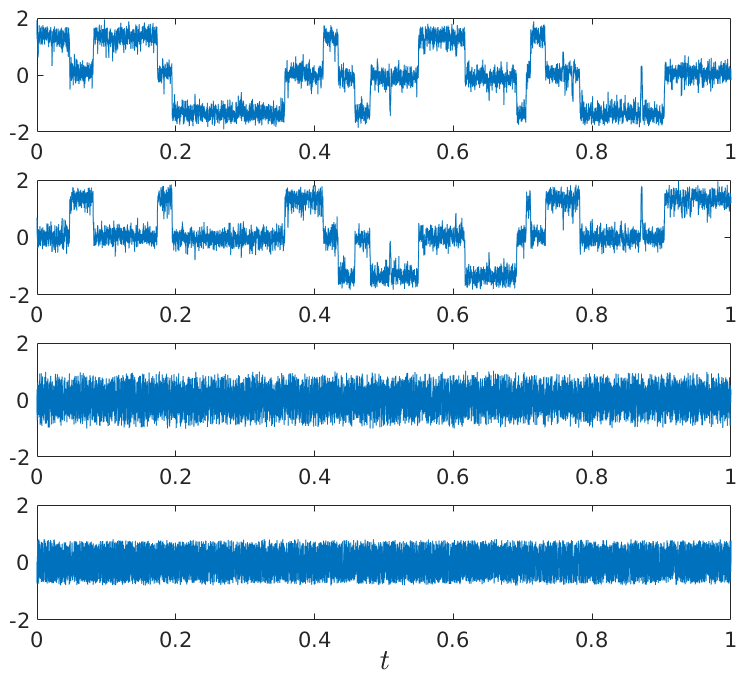}
    \end{minipage}
    \caption{a) Original data set. b) Projection onto the TICA coordinates. Only the first two variables corresponding to the dominant eigenvalues exhibit metastable behavior.}
    \label{fig:TICA}
\end{figure}

\paragraph{$n$-butane.}

Let us now consider a more complex problem, namely the $ n $-butane molecule. It is well-known that the slow dynamics depend mainly on the dihedral angle $ \vartheta $, see Figure~\ref{fig:Butane}. This system was analyzed in \citet{KKS16} using standard EDMD. Since applying EDMD to the full position space is infeasible due to the curse of dimensionality, it was applied to the dihedral angle only. That is, we reduced the originally 42-dimensional (if all atoms are considered) or 12-dimensional (if only the carbon atoms are considered) problem to a one-dimensional problem. However, the slow coordinates are in general unknown a priori and the prime reason why we want to compute eigenfunctions in the first place. Using the kernel-based approaches described above, it is now possible to consider the full state space. That is, we apply our methods directly to the $(x,y,z)$ coordinates of the atoms instead of relying on domain knowledge.

\begin{figure}[htb]
    \centering
    \begin{minipage}[t]{0.49\textwidth}
        \centering
        \subfiguretitle{a)}
        \vspace*{0.8ex}
        \includegraphics[width=0.8\textwidth]{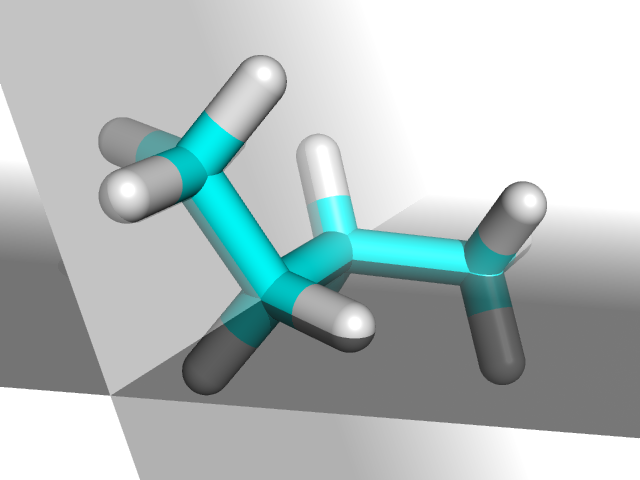}
    \end{minipage}
     \begin{minipage}[t]{0.49\textwidth}
        \centering
        \subfiguretitle{b)}
        \includegraphics[width=0.95\textwidth]{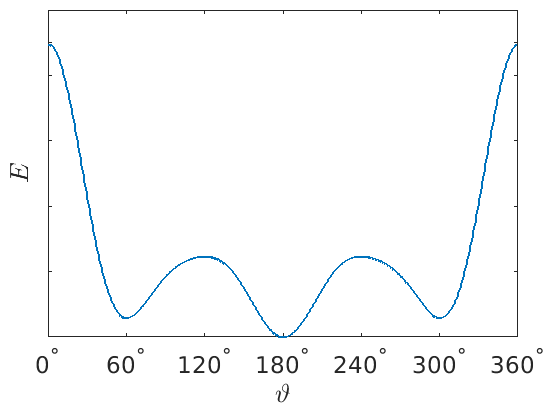}
    \end{minipage}
    \caption{a) Butane molecule and its dihedral angle, given by the angle between the two gray planes. b) Energy $ E $ as a function of the dihedral angle $ \vartheta $. For details, see \cite{KKS16}.}
    \label{fig:Butane}
\end{figure}

We choose a Gaussian kernel with $ \sigma = \frac{1}{\sqrt{2}} $ and select $ 10000 $ data points from one long trajectory computed with Amber~\citep{Amber15}. The lag time $ \tau $ is $ 200 \, \mathrm{fs} $. In order to remove the influence of translations and rotations of the butane molecule, we align the atoms by minimizing the root-mean-square deviation (RMSD) of the four carbon atoms. We also neglect the hydrogen atoms for the subsequent kernel EDMD analysis, the state space $ \mathbb{X} $ is thus 12-dimensional. The second and third eigenfunction of the Perron--Frobenius operator with respect to the invariant density are shown in Figure~\ref{fig:Butane TO}, the first one is omitted since it is, as expected, constant. There is a spectral gap between the third and fourth eigenvalue. We plot the values of the two eigenfunctions evaluated at all data points $ x_i $ as a function of the dihedral angle~$ \vartheta $. The eigenfunctions parametrize the dihedral angle and clearly show the expected three metastable sets around $ \vartheta = 60^\circ $, $ \vartheta = 180^\circ $, and $ \vartheta \approx 300^\circ $, corresponding to the anti and Gauche conformations. The configurations in full space corresponding to these identified metastable sets are shown in Figure~\ref{fig:Butane configurations}. It can clearly be seen that the configurations depend mainly on the dihedral angle.

\begin{figure}[htb]
    \centering
    \begin{minipage}{0.45\textwidth}
        \centering
        \subfiguretitle{a) $ \lambda_2 = 0.996 $}
        \includegraphics[width=0.95\textwidth]{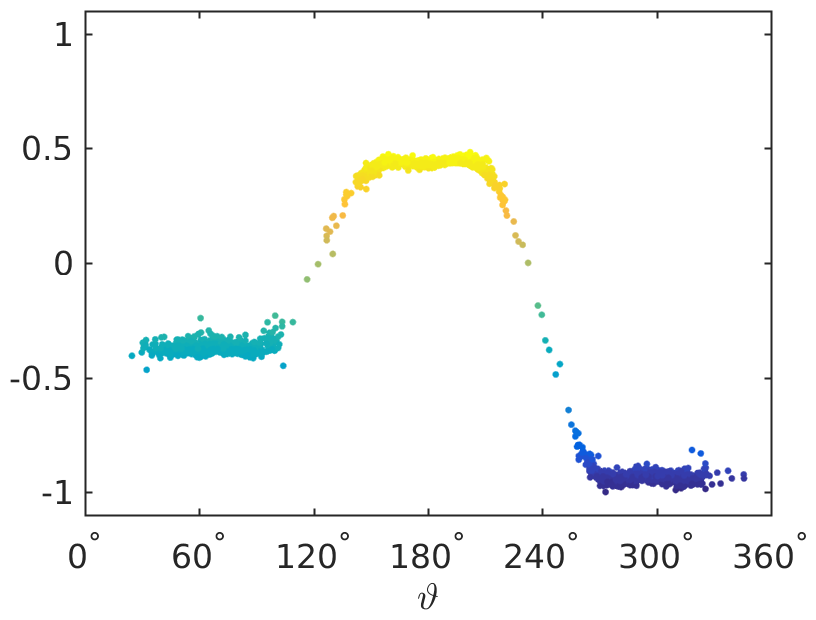}
    \end{minipage}
     \begin{minipage}{0.45\textwidth}
        \centering
        \subfiguretitle{b) $ \lambda_3 = 0.995 $}
        \includegraphics[width=0.95\textwidth]{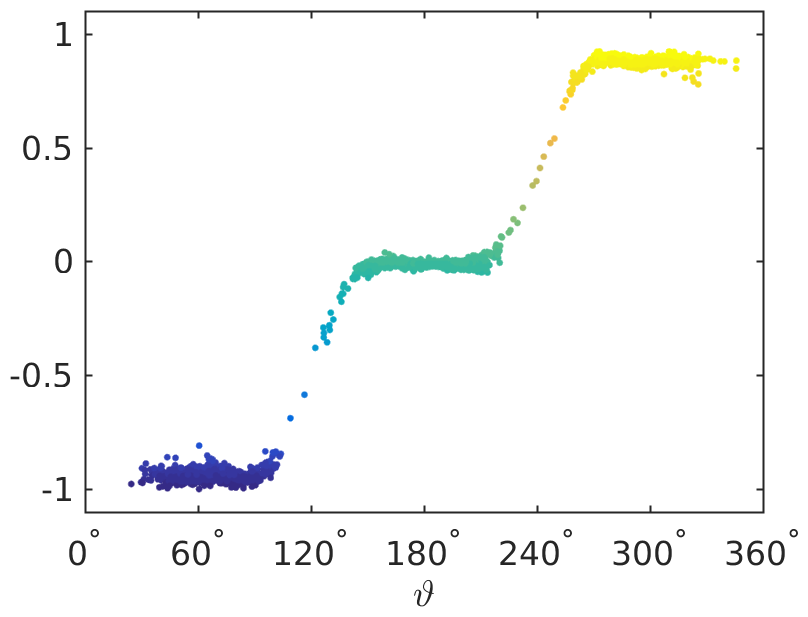}
    \end{minipage}
    \caption{Second and third eigenfunction of the operator $ \mathcal{T}_k $ associated with the butane molecule. The data points $ x_i $ were extracted from one long trajectory. The eigenfunctions are plotted as a function of the dihedral angle $ \vartheta $, which is the known reaction coordinate.}
    \label{fig:Butane TO}
\end{figure}

\begin{figure}[htb]
    \centering
     \begin{minipage}{0.32\textwidth}
        \centering
        \subfiguretitle{a) $ \vartheta \approx 60^\circ $}
        \includegraphics[width=0.95\textwidth]{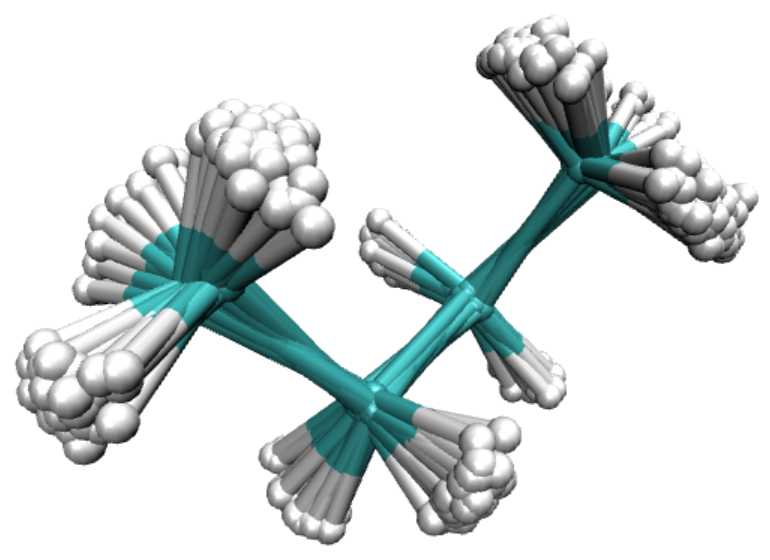}
    \end{minipage}
     \begin{minipage}{0.32\textwidth}
        \centering
        \subfiguretitle{b) $ \vartheta \approx 180^\circ $}
        \includegraphics[width=0.95\textwidth]{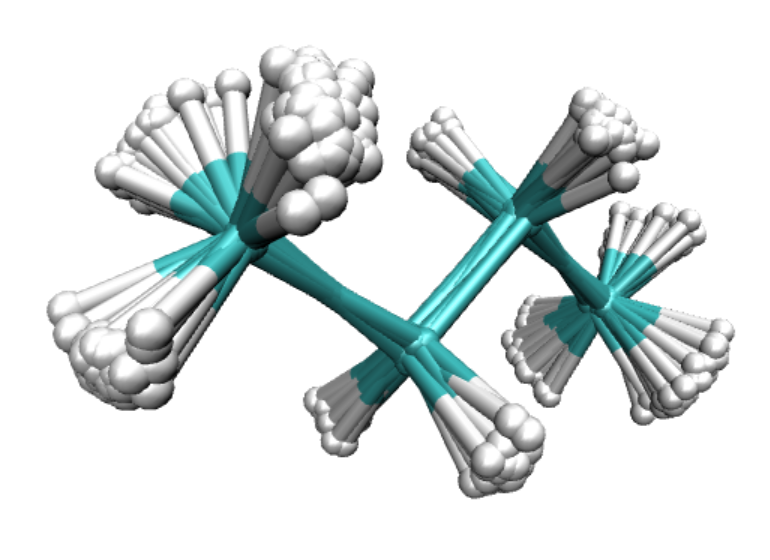}
    \end{minipage}
     \begin{minipage}{0.32\textwidth}
        \centering
        \subfiguretitle{c) $ \vartheta \approx 300^\circ $}
        \includegraphics[width=0.95\textwidth]{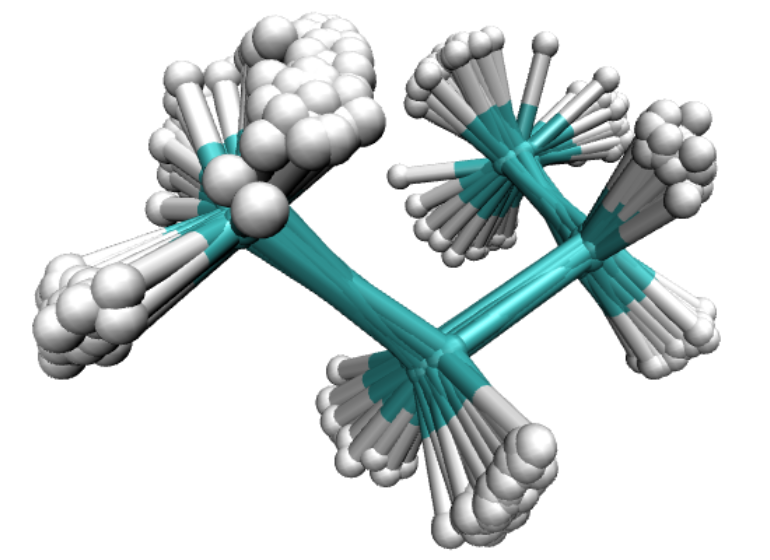}
    \end{minipage}    
    \caption{Overlay of 50 butane configurations taken from the trajectory corresponding to each of the three identified metastable sets.}
    \label{fig:Butane configurations}
\end{figure}

\subsection{Movie Data} % frames 13 through 513 from Pendulum.mpg

This section demonstrates a dimension reduction of high-dimensional time-series data using the proposed method. Consider a simple movie showing a moving pendulum.\!\footnote{\href{http://www.youtube.com/watch?v=MpzaCCbX-z4}{ScienceOnline: The Pendulum and Galileo.}} We want to analyze this data set using the eigendecomposition of $\ko[k]$ for a Gaussian kernel $k$ (corresponding to kernel EDMD). To this end, we convert each $ 576 \times 720 $ RGB video frame to a grayscale intensity image---all intensities are between 0 and 1---and define a kernel $ k(x, y) = \exp\left(-\tfrac{1}{2 \sigma^2} \norm{x - y}_F\right) $, with $ \sigma^2 = 500 $. Here, $ \norm{\cdot}_F $ denotes the Frobenius norm. It would also be possible to use the RGB signal directly, e.g., by defining $ k_\text{RGB}(x, y) = k(x_R, y_R) + k(x_G, y_G) + k(x_B, y_B) $, i.e., each primary color is compared separately. The video comprises $ 501 $ frames so that $ X, Y \in \R^{576 \times 720 \times 500} $. That is, the data sets are now tensors of order three. Analogously, we could reshape the snapshot matrices into vectors. We choose the regularization parameter $ \varepsilon = 0.05 $. Thus, for our analysis, we have to solve the eigenvalue problem $(\gram[XX] + \varepsilon \ts \id_n)^{-1} \gram[YX] \ts v = \lambda \ts v $ to obtain eigenfunctions of $\ko[k]$.

\begin{figure}[thbp]
    \centering
    \begin{minipage}{0.32\textwidth}
        \centering
        \subfiguretitle{a) Frame 1}
        \includegraphics[width=0.95\textwidth]{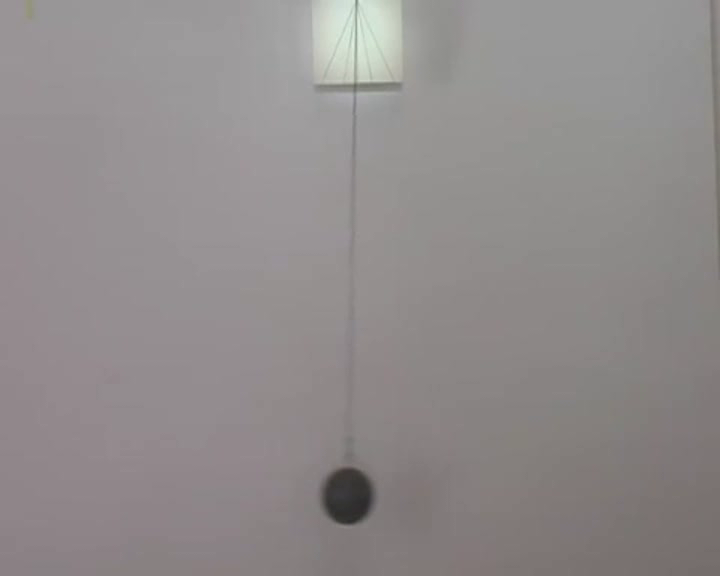}
    \end{minipage}
    \begin{minipage}{0.32\textwidth}
        \centering
        \subfiguretitle{b) Frame 13}
        \includegraphics[width=0.95\textwidth]{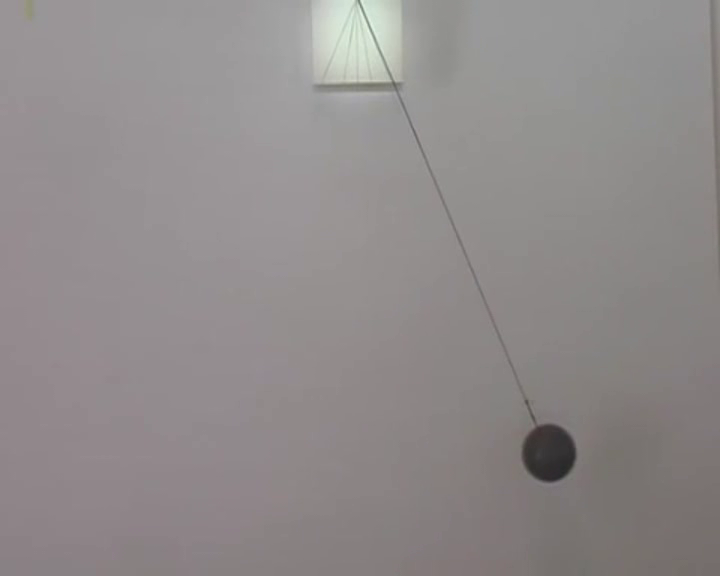}
    \end{minipage}
    \begin{minipage}{0.32\textwidth}
        \centering
        \subfiguretitle{c) Frame 36}
        \includegraphics[width=0.95\textwidth]{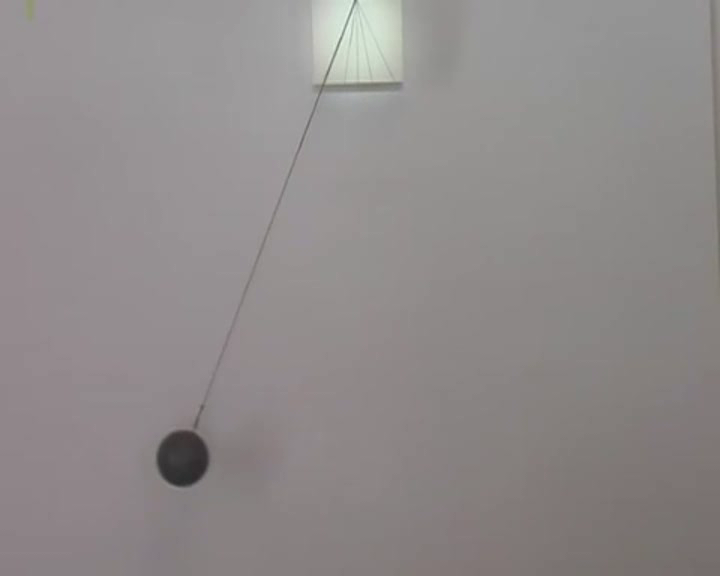}
    \end{minipage} \\[1em]
    \begin{minipage}{0.99\textwidth}
        \centering
        \subfiguretitle{d)}
        \includegraphics[width=\textwidth]{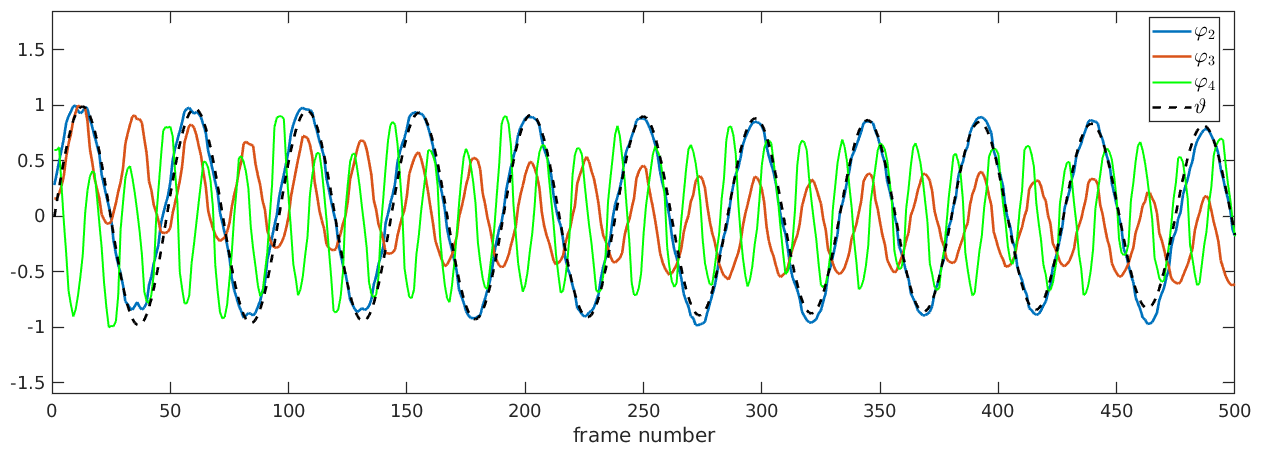}
    \end{minipage}\\[1em]
    \begin{minipage}{0.32\textwidth}
        \centering
        \subfiguretitle{e) $ \varphi_2 $}
        \includegraphics[width=0.95\textwidth]{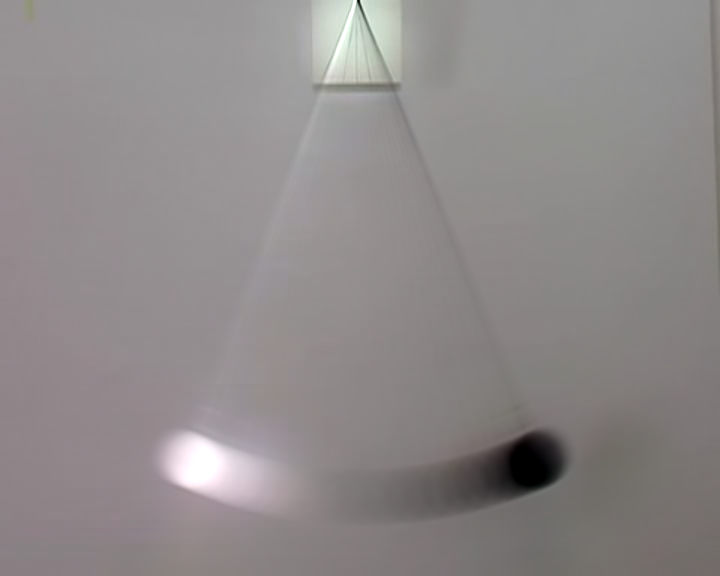}
    \end{minipage}
    \begin{minipage}{0.32\textwidth}
        \centering
        \subfiguretitle{f) $ \varphi_3 $}
        \includegraphics[width=0.95\textwidth]{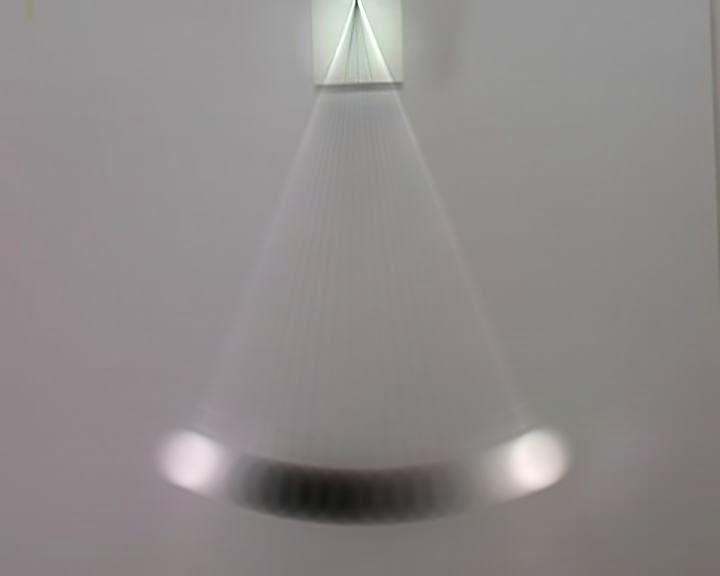}
    \end{minipage}
    \begin{minipage}{0.32\textwidth}
        \centering
        \subfiguretitle{g) $ \varphi_4 $}
        \includegraphics[width=0.95\textwidth]{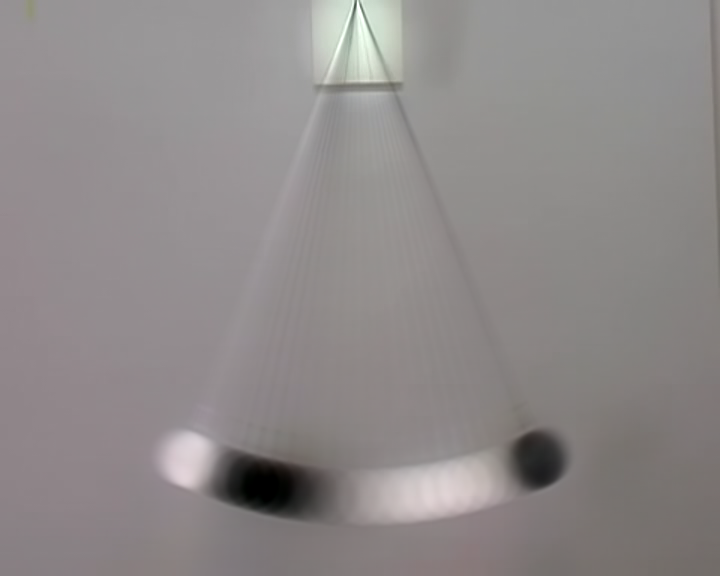}
    \end{minipage}
    \caption{a)~No angular displacement. b)~Maximum displacement right-hand side. c)~Maximum displacement left-hand side. d)~Values of the normalized eigenfunctions $ \varphi_2 $, $ \varphi_3 $, and $ \varphi_4 $ for each frame. The eigenfunctions encode the frequency of the pendulum. The frames 13 and 36 correspond to the first maximum and minimum of the eigenfunction $ \varphi_2 $. The period of $ \varphi_3 $ is twice the period of $ \varphi_2 $. The black dashed line shows the angular displacement $ \vartheta $ (rescaled for the sake of comparison) obtained by a numerical simulation of the pendulum. The dominant eigenfunction parametrizes the angular displacement. e--g) Minima of the leading eigenfunctions computed using gradient descent, see \cite{KPS18} for details. The video snapshots are reproduced with the kind permission of \emph{ScienceOnline}.}
    \label{fig:Pendulum}
\end{figure}

The values of the resulting nontrivial dominant eigenfunctions $ \varphi_2 $, $ \varphi_3 $, and $ \varphi_4 $ evaluated for each frame and the associated single snapshot data summarizations are shown in Figure~\ref{fig:Pendulum}. The first nontrivial eigenfunction encodes the frequency of the pendulum and the second eigenfunction twice the frequency. As a result, we could now sort the frames according to the angular displacement of the pendulum using the eigenfunctions. The information encoded in the eigenfunctions can be visualized using gradient-based optimization techniques (i.e., by finding states that minimize or maximize a given eigenfunction) as described in \cite{KPS18}. For this simple example, we used the raw video data. For more complex systems, preprocessing steps might be beneficial, e.g., mean subtraction, Sobel edge detection, or more sophisticated feature detection approaches such as SIFT or HOG \citep{BRF10}. In this way, it would be possible to track features of images over time.

\subsection{Text Data}

In this section, we show how the eigendecomposition of the kernel Perron--Frobenius operator with respect to the invariant density, denoted by $ \mathcal{T}_k $, can be used for non-vectorial data. Consider the following scenario: Given a collection of text documents, we first erase all words not contained in a predefined vocabulary. The vocabulary in this example, shown in Table~\ref{tab:Keywords}, was chosen based on keywords that appeared repeatedly in the text. Of the remaining words, one word (denoted by $ y_i $) following another (denoted by $ x_i $) is considered to be its time-evolved version or successor. This defines a discrete dynamical system. The lists of all such words $ x_i $ and $ y_i $ are denoted by $ \mathbf{X} $ and $ \mathbf{Y} $, respectively. We collect 1000 word pairs  from news articles. As an example, let us parse the following sentence:
\blockquote{ \small
Macron’s announcement Wednesday was the latest attempt by a government to find ways to handle the worldwide spread of disinformation on social media -- ``fake news'', as U.S. \textbf{President} Donald Trump calls it. His plan would allow judges to block a website or a user account, in particular during an \textbf{election}, and oblige \textbf{internet} platforms to publish the names of those behind sponsored contents.\!\footnote{\href{https://www.reuters.com/article/us-france-macron-fakenews/french-opposition-twitter-users-slam-macrons-anti-fake-news-plans-idUSKBN1EU161}{Reuters: French opposition, Twitter users slam Macron's anti-fake-news plans}.}
}
Here, we would first remove all words not contained in the vocabulary and thus obtain the word pairs (``president'', ``election'') and (``election'', ``internet''). Typically, the same word or related words are used several times within one article, but words related to other topics are rarely mentioned. Since we consider the sequence of articles as one long document\footnote{Parts of the same articles are used several times to increase the size of the data set, this is thus a synthetic example, mainly to illustrate the concept.}, transitions occur, for instance, when one article ends and the next one about a different topic starts, when different topics are mixed, or when words such as \emph{state} or \emph{cell} are used in a different context. These are the rare transitions that are similar to the jumps between the wells in the molecular dynamics example. Although this is a slightly artificial example, it illustrates how to extend transfer operator approaches to new domains where only a similarity measure given by a kernel is available.

\begin{table}[t!]
    \centering
    \caption{Predefined set of keywords.}
    \def\arraystretch{1.2}
    \begin{tabular}{lllll}
        \hline
        browser & cell & computer & damage & department \\
        disease & e-mail & election & hurricane & internet \\
        midterm & president & rain & science & state \\
        stem & storm & tablet & therapy & weather \\
        \hline
    \end{tabular}
    \label{tab:Keywords}
\end{table}

\begin{figure}[th] 
    \centering
    \includegraphics[width=0.53\textwidth]{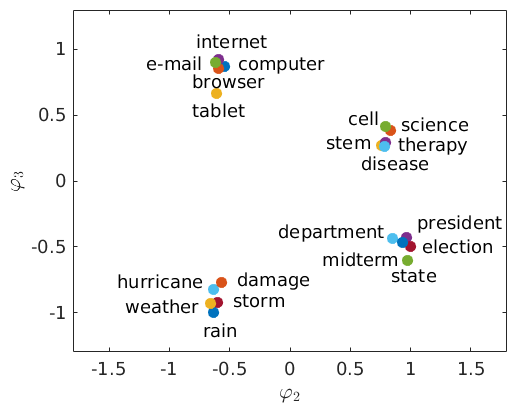}
    \caption{Clustering based on the dominant eigenfunctions. Our method identified four topic clusters: information technology, medicine, weather, and politics.}
    \label{fig:Keyword clustering}
\end{figure}

We evaluate the Gram matrices $ \gram[XY] $ and $ \gram[XX] $ and compute eigenfunctions of the operator~$ \mathcal{T}_k $. Here, $ [\gram[XY]]_{ij} = k(x_i, y_j) $, where $ x_i $ is the $ i $th word in $ \mathbf{X} $ and $ y_j $ the $j$th word in~$ \mathbf{Y} $. Correspondingly, $ \gram[XX] $ is the standard Gram matrix. Moreover, $ k = \exp\left(-\frac{1}{2 \sigma^2} k_s(\cdot, \cdot)^2\right) $, where $ k_s $ is the text kernel proposed in \citet{Lodhi02}\footnote{We use the \href{https://github.com/mmadry/string_kernel}{String Kernel Software} implementation.} and $ \sigma^2 = 0.2 $.
Given two words, the text kernel compares the substrings they contain, i.e., the higher the number of identical substrings, the higher the similarity. Substrings do not necessarily have to be contiguous, but matching substrings that are far apart are assigned a lower weight. We obtain, for instance, $ k_s(\text{``department''}, \text{``president''}) \approx 0.47 $ (relatively high similarity due to the same ending) and $ k_s(\text{``department''}, \text{``science''}) = 0.07 $ (almost no similarity). We compute again the leading nontrivial eigenfunctions $ \varphi_2 $ and $ \varphi_3 $ and use the the eigenfunctions as coordinates. The results are shown in Figure~\ref{fig:Keyword clustering}. Note that the words are not clustered based on string kernel similarity but on proximity in the document collection. Words that often occur together are grouped into clusters (given one word of such a cluster, the next word will most likely be in the same cluster). For this simple example, it would also have been possible to assign each word a distinct number and to generate a Markov state model by approximating the transition probabilities between words. The eigenvectors of the Markov matrix would then lead to a similar clustering. The text kernel, however, takes into account string similarity. This is important to account, for example, for grammatical variations reflected in word form (\emph{green} vs.\ \emph{greener}) and misspellings (\emph{love} vs.\ \emph{loove}) without necessarily resorting to lemmatizing, stemming, or other normalization techniques. Another possibility here would be to design linguistically informed string kernels. In German for example, a \emph{Visumantrag} (visa application) is more similar to \emph{Antrag} (application) than to \emph{Visum}. A string kernel taking this into account would instantly be reflected in the word clusters discovered by our method, which could never be achieved when using a pure Markov state model.

\section{Conclusion}
\label{sec:conclusion}

We extended transfer operator theory to reproducing kernel Hilbert spaces and illustrated similarities with the conditional mean embedding framework. While the conventional transfer operator propagates densities, the kernel transfer operator can be viewed as an operator that propagates embedded densities. Moreover, we have highlighted relationships between the covariance and cross-covariance operator based methods to obtain empirical estimates of the kernel transfer operators and other well-known methods, e.g., TICA and EDMD, for the approximation of transfer operators developed by the dynamical systems, molecular dynamics, and fluid dynamics communities. The eigendecompositions of kernel transfer operators provide a powerful tool for analyzing nonlinear dynamical systems. One main benefit of purely kernel-based methods is that these methods can be applied to non-vectorial data such as strings or graphs. We demonstrated the efficiency and versatility of these methods using guiding examples as well as simple molecular dynamics applications, video data, and text data. 

Our future work includes applying the proposed methods to more realistic data sets, in particular more complicated video data, potentially in combination with machine learning based preprocessing approaches. The main remaining theoretical question is the convergence of the kernel transfer operators to the actual transfer operators.
In \cite{KBSS18}, it was shown that, using the Mercer feature space representation, kernel transfer operators can be interpreted as Galerkin approximations of their analytical counterparts. Combining this with the EDMD convergence results obtained in \cite{KoMe18}, it might be possible to show also convergence for kernels with infinite-dimensional feature spaces. Furthermore, the influence of the kernel itself, the regularization parameter, and the number of test points on the accuracy of the eigenfunction approximations is not yet clear. Another extension of the framework presented within this paper would be to use singular value decompositions instead of eigenvalue decompositions. The resulting methods could then also be applied to problems where the spaces $ \inspace $ and $ \outspace $ are different.

\section*{Acknowledgements}

This research has been partially funded by Deutsche Forschungsgemeinschaft (DFG) through grant CRC 1114 \emph{``Scaling Cascades in Complex Systems''}. Krikamol Muandet acknowledges fundings from the Faculty of Science, Mahidol University and the Thailand Research Fund (TRF). We would like to thank the reviewers for their helpful comments.

\bibliographystyle{unsrtnat}
\bibliography{cme_to}

\end{document}

%% file: RKHS-TO-macros.tex
\newcommand{\ts}{\hspace*{0.1em}} % thin space

\newcommand{\subfiguretitle}[1]{{\scriptsize{#1}} \\[0.25ex]}
\newcommand{\R}{\mathbb{R}}                                     % real numbers
\newcommand{\innerprod}[2]{\left\langle #1,\, #2 \right\rangle} % scalar product
\newcommand{\dd}{\mathrm{d}}                                    % for integrals and differential equations
\newcommand{\pp}[1]{\mathbb{#1}}                                % probability measure
\providecommand{\abs}[1]{\left\lvert #1 \right\rvert}           % absolute value
\providecommand{\norm}[1]{\left\lVert #1 \right\rVert}          % norm

\newcommand\xqed[1]{\leavevmode\unskip\penalty9999 \hbox{}\nobreak\hfill \quad\hbox{#1}}
\newcommand{\exampleSymbol}{\xqed{$\blacktriangle$}} % end example symbol

\newcommand{\inspace}{\mathbb{X}} % the input space X
\newcommand{\outspace}{\mathbb{Y}} % the output space Y

\newcommand{\id}{I} % identity matrix
\newcommand{\idop}{\mathcal{I}} % identity operator

\newcommand{\ebd}[1][]{% integral operator defined by a kernel
   \ifthenelse{\equal{#1}{}}{\mathcal{E}}{\mathcal{E}_{#1}}}

\newcommand{\cme}{\mathcal{U}_{\scriptscriptstyle Y \mid X}} % conditional mean embedding
\newcommand{\ecme}{\widehat{\mathcal{U}}_{\scriptscriptstyle Y \mid X}} % empirical conditional mean embedding operator

\newcommand{\pf}[1][]{% Perron-Frobenius operator with optional subscript
   \ifthenelse{\equal{#1}{}}{\mathcal{P}}{\mathcal{P}_{#1}}}
\newcommand{\epf}[1][]{% empirical Perron-Frobenius operator with optional subscript
   \ifthenelse{\equal{#1}{}}{\widehat{\mathcal{P}}}{\widehat{\mathcal{P}}_{#1}}}
\newcommand{\ko}[1][]{% Koopman operator with optional subscript
   \ifthenelse{\equal{#1}{}}{\mathcal{K}}{\mathcal{K}_{#1}}}
\newcommand{\eko}[1][]{% empirical Koopman operator with optional subscript
   \ifthenelse{\equal{#1}{}}{\widehat{\mathcal{K}}}{\widehat{\mathcal{K}}_{#1}}}

\newcommand{\cov}[1][]{\mathcal{C}_\mathit{\scriptscriptstyle #1}} % covariance operators
\newcommand{\ecov}[1][]{\widehat{\mathcal{C}}_\mathit{\scriptscriptstyle #1}} % empirical estimates of covariance operators
\newcommand{\gram}[1][]{G_\mathit{\scriptscriptstyle #1}} % gram matrices

\DeclareMathOperator{\mspan}{span}

\bibpunct{(}{)}{,}{a}{,}{,}

\mathtoolsset{centercolon}

\allowdisplaybreaks